\documentclass[12pt,reqno]{amsproc}%
\usepackage{amsfonts}
\usepackage{amsmath}
\usepackage{amssymb}
\usepackage{graphicx}%
 \theoremstyle{plain}
\newtheorem{theorem}{Theorem}[section]

\newtheorem{corollary}[theorem]{Corollary}

\newtheorem{definition}[theorem]{Definition}
\newtheorem{example}[theorem]{Example}

\newtheorem{lemma}[theorem]{Lemma}

\newtheorem{proposition}[theorem]{Proposition}
\newtheorem{remark}[theorem]{Remark}

\numberwithin{equation}  {section}

\begin{document}

\title{Similarity degree of a class of C$^*$-algebras}

\author{Wenhua Qian}
\address{Wenhua Qian \\
	Research Center for Operator Algebras, Department of Mathematics \\
	East China Normal University(Minhang Campus) \\
	Minhang District, Shanghai 200241, China; \ Email: whqian86@163.com}
\author{Junhao Shen}
\address{Junhao Shen \\
        Demartment of Mathematics and Statistics \\
         University of New Hampshire\\
         Durham, NH 03824, US; \  Email:  Junhao.Shen@unh.edu}

\begin{abstract}
Suppose that $\mathcal M$ is a countably decomposable type II$_1$ von Neumann algebra and $\mathcal A$ is a separable, non-nuclear, unital C$^*$-algebra. We show  that, if $\mathcal M$ has Property $\Gamma$, then
the similarity degree of $\mathcal M$ is less than or equal to $5$. If   $\mathcal A$ has Property c$^*$-$\Gamma$, then the similarity degree of $\mathcal A$ is   equal to $3$. In particular, the similarity degree of a $\mathcal Z$-stable, separable, non-nuclear, unital C$^*$-algebra is  equal to $3$.
\end{abstract}

\subjclass[2000]{Primary     46L10; Secondary 46L05}
\keywords{Property $\Gamma$, Similarity problem, Similarity degree}


\maketitle

\section{Introduction}
Kadison's Similarity Problem for a  C$^*$-algebra $\mathcal A$ in
\cite{Ka1} asks whether every
    bounded  representation $\rho$ of   $\mathcal A$ on a Hilbert
space $  H$ is similar to a $*$-representation. i.e. whether
there exists an invertible operator $T$ in $B(H)$, such
that $T\rho(\cdot)T^{-1}$ is a $*$-representation of $\mathcal A$.

 In
\cite {C3}, Christensen proved that every irreducible bounded
representation of a C$^*$-algebra on a Hilbert space is similar to a
$*$-representation. He also showed that
 every
  non-degenerate  bounded  representation of a nuclear C$^*$-algebra
is similar to a $*$-representation (also see  \cite{B1}).

In \cite{H1}, Haagerup showed that every
     cyclic (or finitely cyclic) bounded representation of a C$^*$-algebra   on a Hilbert
space  is similar to a $*$-representation. From this, he concluded
that, if $\mathcal A$ is a C$^*$-algebra that has no tracial states,
then every
  non-degenerate  bounded representation of $\mathcal A$ is similar to
a $*$-representation. It was  also shown in \cite{H1} that a
non-degenerate bounded  representation $\rho$ of a C$^*$-algebra
$\mathcal A$ on a Hilbert space $ H$ is similar to a
$*$-representation if and only if $\rho$ is completely bounded (also
see \cite{Had}, \cite{W}).

From Haagerup's results, it follows that  Kadison's
Similarity Problem   remains only open for C$^*$-algebras with
tracial states. Since a type II$_1$ von Neumann algebra always has
tracial states, it is natural to consider Kadison's Similarity
Problem for von Neumann algebras of type II$_1$.  In \cite {C4},
Christensen showed that Kadison's Similarity Problem for type II$_1$
factors with Property $\Gamma$ has an affirmative answer.

 In order to study Kadison's Similarity Problem,
Pisier in \cite{Pi2} introduced a  powerful new concept, similarity
degree of a unital C$^*$-algebra, as follows. {\em Suppose
$\mathcal{A}$ is a unital C$^*$-algebra. It has finite simiarity
degree if there is $\alpha >0$ such that for some constant $k$
(depending on $\alpha$) we have,  for  every bounded unital
representation  $\phi$ of $\mathcal{A} $ on a Hilbert space
$  H$,
$$\Vert \phi \Vert _{cb}\leq k \Vert \phi \Vert ^{\alpha}.$$
The infimum of the   numbers $\alpha$ (if exists) for which this
holds is defined to be the similarity degree of $\mathcal{A}$. We
denote it by $d(\mathcal{A})$. If there is no such pair
$(\alpha,k)$, we define $d(\mathcal{A})=\infty$.} It was shown in
\cite{Pi2} that Kadison's Similarity Problem   for a unital
C$^*$-algebra $\mathcal{A}$ has an affirmative answer if and only if
$d(\mathcal{A})<\infty$.

The following is a list of some recent results on the similarity
degrees of infinite dimensional unital C$^*$-algebras. (We have no intention to make the list complete.)
\begin{enumerate}
\item [(i)] $d(\mathcal{A})=2$ if and only if $\mathcal{A}$ is nuclear. (\cite{B1}, \cite{C1},
\cite{Pi3})

\item [(ii)]  If $\mathcal{A}=B(H)$ for some Hilbert space $H$, then
$d(\mathcal{A})=3$.
(\cite{H1}, \cite{Pi4})

\item [(iii)]  If $\mathcal{A}$ is a type II$_{1}$ factor with Property $\Gamma$,   $d(\mathcal{A}) \le
5$, (\cite{Pi4}).  This result was later improved in \cite{C2} to
$d(\mathcal{A}) = 3$.

\item [(iv)] If $\mathcal{A}$ is a minimal tensor product of two C*-algebras, one of which is nuclear and contains matrices of any
order, then $d(\mathcal{A}) \le 5$. (\cite{Po})

\item [(v)]  If $\mathcal{A}$ is $\mathcal{Z}$-stable, then $d(\mathcal{A}) \leq
5$. (\cite{JW})

\item [(vi)] If   every II$_{1}$ factor $*$-representation of a separable
C*-algebra $\mathcal{A}$ has Property $\Gamma$, then $d(\mathcal{A})
\leq 11$. (\cite{HS})
\end{enumerate}

In the paper, we are interested in Kadison's Similarity Problem for
type $II_1$ von Neumann algebras with Property $\Gamma$. Recall the
definition of Property $\Gamma$ for a type II$_{1}$ von Neumann
algebra from \cite{QS}. {\em Suppose $\mathcal{M}$ is a type
II$_{1}$ von Neumann algebra with a predual $\mathcal M_{\sharp}$.
Suppose   $\sigma (\mathcal M, \mathcal M_\sharp)$ is the weak-$*$
topology on $\mathcal M$ induced from $\mathcal M_\sharp$. We say
that $\mathcal{M}$ has Property $\Gamma$ if and only if $\forall \
a_{1}, a_{2}, \dots, a_{k} \in \mathcal{M}$ and $\forall \ n\in \Bbb
N$, there exist a partially ordered set $\Lambda$ and a family of projections $$\{ p_{i \lambda}: 1\le i\le n;
\lambda \in \Lambda \}\subseteq \mathcal{M}$$ satisfying
\begin{enumerate}
\item [(i)] For each $\lambda \in \Lambda$,   $p_{1 \lambda}, p_{2 \lambda}, \dots,
p_{n \lambda}$ are mutually orthogonal equivalent projections in $\mathcal
M$ with sum $I$.
\item [(ii)] For each $1\le i\le n$ and $1\le j\le k$,
$$
\lim_{ \lambda } (p_{i \lambda}a_{j}-a_{j}p_{i
 \lambda})^*(p_{i \lambda}a_{j}-a_{j}p_{i \lambda}) =0 \qquad \text {in $\sigma(\mathcal
M, \mathcal M_\sharp)$ topology.}$$
\end{enumerate}}

The first result we obtain in the paper is the following equivalent definition of Property $\Gamma$ for a countably decomposable type II$_1$ von Neumann algebra, which gives an analogue of Murray and von Neumann's definition of Property $\Gamma$ for a type II$_1$ factor.

  \vspace{0.2cm}

 {Proposition  \ref{3.5.1}.} \ {\em
Let $\mathcal{M}$ be a countably decomposable type II$_{1}$ von Neumann algebra  and   $\mathcal{Z}_{\mathcal{M}}$ be the center of $\mathcal{M}$. Suppose $\tau$ is a center valued trace from $\mathcal{M}$ to $\mathcal{Z}_{\mathcal{M}}$ such that $\tau(a)=a$  for all $a\in \mathcal {Z}_{\mathcal{M}}$. Then the following are equivalent.
\begin{enumerate}
\item [(1)] $\mathcal{M}$ has Property $\Gamma$.

 \item [(2)]
 There exist a positive inter $n_0\ge 2$ and a faithful normal tracial state $\rho$ on $\mathcal M$ such that, {\em for any
 $\epsilon >0$ and elements $a_{1}, a_{2}, \dots, a_{k} \in \mathcal{M}$, there exists a family of orthogonal equivalent projections
   $p_1,\ldots, p_{n_0}$ in $\mathcal M$ with sum $I$ satisfying  $\Vert p_ia_{j}-a_{j}p_i \Vert_{2} < \epsilon$  for all $i=1,\ldots, n_0$ and  $j=1,2, \dots, k$, where $\Vert \cdot \Vert_{2}$ is the $2$-norm induced by $\rho$.}
 \item  [(3)]
 There exists a faithful normal tracial state $\rho$ on $\mathcal M$ such that, {\em for any
 $\epsilon >0$ and elements $a_{1}, a_{2}, \dots, a_{k} \in \mathcal{M}$, there exists a unitary $u \in \mathcal{M}$ satisfying (i) $\tau (u)=0$ and (ii) $\Vert ua_{j}-a_{j}u \Vert_{2} < \epsilon$  for all $j=1,2, \dots, k$, where $\Vert \cdot \Vert_{2}$ is the $2$-norm induced by $\rho$.}

 \end{enumerate}

 }

  \vspace{0.2cm}

Next we are able to obtain  an upper bound for the similarity degree of a countably decomposable type II$_1$ von Neumann algebra with Property $\Gamma$, which extends Theorem 13 in \cite{Pi4}.

  \vspace{0.2cm}

 {Theorem \ref{mainthm1}.} \ {\em If $\mathcal{M}$ is
a countably decomposable type  II$_{1}$ von Neumann algebra with
Property $\Gamma$, then $d(\mathcal{M}) \leq 5$. }

\vspace{0.2cm}

From Theorem \ref{mainthm1}, it follows that Kadision's Similarity Problem for a countably decomposable type II$_1$ von Neumann algebra with Property $\Gamma$ has an affirmative answer.

The last main result we obtained in the paper is the following computation of the similarity degree for a class of C$^*$-algebras.

\vspace{0.2cm}

 {Theorem \ref{5.3}. } \ {\em Suppose $\mathcal{A}$ is a separable unital C$^*$-algebra
satisfying
\begin{enumerate} \item[] Condition (G):
if $\pi$ is a unital $*$-representation of $\mathcal{A}$ on a Hilbert space
$  H$ such that $\pi(\mathcal{A})''$ is a type  II$_{1}$
factor, then $\pi(\mathcal{A})''$ has Property $\Gamma$, where
$\pi(\mathcal{A})''$ is the von Neumann algebra generated by
$\pi(\mathcal{A})$ in $B( H)$.\end{enumerate} Then
$d(\mathcal{A}) \leq 3.$ Moreover, if $\mathcal A$ is non-nuclear, then $d(\mathcal{A}) = 3.$}

\vspace{0.2cm}

As a corollary, we get that if $\mathcal{A}$ is a minimal  tensor product of
two separable unital C*-algebras, one of which is nuclear and has no finite
dimensional $*$-representations, then $d(\mathcal{A}) \le 3 $. In particular, the similarity degree of a $\mathcal Z$-stable, non-nuclear, separable, unital C$^*$-algebra is   equal to $3$. This
gives a generalization of earlier results in \cite{Po}, \cite{JW},
\cite{HS}.

The paper is organized as follows. In section 2, we introduce
notation and preliminaries. In section 3, we will give   an
equivalent definition of Property $\Gamma$ for countably decomposable type  II$_{1}$ von
Neumann algebras. In section 4, we show that
if $\mathcal{M}$ is a countably decomposable type II$_{1}$ von Neumann algebra Property $\Gamma$, then $d(\mathcal{M}) \leq
5$. By the result in Section 4, we prove in Section 5 that if the type  II$_{1}$ central summand
in the type decomposition of a von Neumann algebra $\mathcal{M}$ is a countably decomposable von Neumann algebra with Property $\Gamma$, then any bounded,  $\sigma(\mathcal M,\mathcal M_{\sharp})$
to  $\sigma(B(H), B(H)_{\sharp})$  continuous, unital homomorphism $\phi: \mathcal{M} \to
B(H)$ is completely bounded and $\Vert \phi \Vert_{cb} \leq \Vert
\phi \Vert^{3}$. As a consequence of this result, we obtain that, if
  a  separable unital C*-algebra
$\mathcal{A}$ has Property c$^*$-$\Gamma$, then $d(\mathcal{A}) \leq 3$. This is the first paper of our series. In our following work in \cite{QHS}, the class of separable unital C$^*$-algebras with Property c$^*$-$\Gamma$ will be applied to show that, if $\mathcal{M}$ is a type II$_1$ von Neumann
algebra with Property $\Gamma$, then $d(\mathcal{M})=3$.

\section{Preliminaries}

\subsection{Similarity length and similarity degree of a unital C$^*$-algebra}
In this subsection, we recall Pisier's similarity length and similarity degree for a unital C$^*$-algebra.
\begin{definition}\label{2.1}
(\cite{Pi1}) A unital operator algebra $\mathcal{A}$ has finite similarity length at most $l \in \mathbb{N}$ if there exists a constant $C$ such that, for any $k \in \mathbb{N}$ and any $x \in M_{k}(\mathcal{A})$, threre exist an $n \in \mathbb{N}$ and scalar matrices $\alpha _{0} \in M_{k,n}(\mathbb{C}), \alpha _{1} \in M_{n}(\mathbb{C}), \dots, \alpha _{l-1} \in $ $M_{n}(\mathbb{C}), \alpha _{l} \in M_{n,k}(\mathbb{C})$ and diagonal matrices $D_{1}, D_{2}, \dots, D_{l} \in M_{n}(\mathcal{A})$ such that
$$x=\alpha _{0}D_{1}\alpha _{1}D_{2}...D_{l}\alpha _{l}$$
and
$$(\prod\limits_{0}^{l} \Vert \alpha _{l} \Vert) (\prod\limits_{1}^{l} \Vert D_{l} \Vert ) \leq C \Vert x \Vert .$$
The length $l(\mathcal{A})$ is defined to be the least possible $l$ for which these conditions are fulfilled.
\end{definition}

\bigskip
It was verified in \cite{Pi1} that the Kadison's Similarity Problem has a positive answer for a unital C$^*$-algebra if and only if the C$^*$-algebra has finite similarity length.

\begin{definition} \label{2.2}
(\cite{Pi2}) Let $\mathcal{A}$ be a unital C$^*$-algebra. We define
the similarity degree of $\mathcal{A}$ to be the infimum of all
positive numbers $\alpha$ (if it exists) for which there is $k>0$
such that, for every bounded unital homomorphism $\rho$ of
$\mathcal{A} $ on a Hilbert space $ H$, we have
$$\Vert \rho \Vert _{cb}\leq k \Vert \rho \Vert ^{\alpha }.$$
 We denote such infimum by $d(\mathcal{A})$. If there are no such pairs $(\alpha,k)$, we define $d(\mathcal{A})=\infty$.
\end{definition}
It was proved in \cite{Pi2} that the similarity degree and the similarity length of a unital C$^*$-algebra (if they are finite) are the same integer.

\subsection{Dual space}
Suppose $\mathcal{A}$ is a unital C$^*$-algebra. Its dual space, the set of all bounded linear functionals on $\mathcal{A}$, is denoted by $\mathcal{A}^{\sharp}$. The second dual space $\mathcal{A}^{\sharp \sharp}$ of $\mathcal{A} $  is isomorphic to $\pi (\mathcal{A})''$, where $\pi$ is the universal representation of $\mathcal{A}$ and $\pi (\mathcal{A})''$ is the von Neumann algebra generated by $\pi (\mathcal{A})$. Thus  $\mathcal{A}^{\sharp \sharp}$  is always viewed as a von Neumann algebra and  $\mathcal{A}$ becomes a C$^*$-subalgebra of $\mathcal{A}^{\sharp \sharp}$ when $\mathcal A$ is canonically embedded into $\mathcal{A}^{\sharp \sharp}$.

  Suppose $\mathcal M$ is a von Neumann algebra with a (unique) predual space $\mathcal M_\sharp$. We will denote by $\sigma(M,M_\sharp)$ the weak-$*$ topology on $\mathcal M$ induced by $\mathcal M_\sharp$.

The following lemma is well-known. (See Theorem 1 in \cite{Ba} or Proposition 1.21.13 in \cite{Sa})
\begin{lemma} \label{2.3}
Suppose  $\mathcal{A}$ is a unital C$^{* }$-algebra and $\phi:\mathcal{A} \to
B(H)$ is a bounded unital homomorphism of $\mathcal A$ on a Hilbert space $H$. Then $\phi$ can be extended to a bounded unital homomorphism $\overline{\phi}:\mathcal{A}^{\sharp \sharp} \to B(H)$ that is $\sigma(\mathcal{A}^{\sharp\sharp}, \mathcal{A}^{\sharp}) \to \sigma(B(H), B(H)_{\sharp})$ continuous (in other words, it is weak-$*$ to weak-$*$  continuous), where $B(H)_{\sharp}$ is the predual of $B(H)$. Moreover  $\Vert \overline{\phi} \Vert = \Vert \phi \Vert$.
\end{lemma}

\subsection{Direct integral}
The concept of direct integral was introduced by von Neumann in \cite{vN}. Here are some results about direct integral that we need in this paper.

\begin{lemma} \label{2.4}
(\cite{KR1}) Suppose $\mathcal{M}$ is a von Neumann algebra acting on a separable Hilbert space $H$. Let $\mathcal{Z}_{\mathcal{M}}$ be the center of $\mathcal{M}$. Then there is a direct integral decomposition of $\mathcal{M}$ relative to $\mathcal{Z}_{\mathcal{M}}$, i.e. there exists a locally compact complete separable metric measure space $(X, \mu)$ such that

\begin{enumerate}
\item[(i)] $H$ is (unitarily equivalent to) the direct integral of $\{ H_{s} : s \in X \}$ over $(X, \mu)$, where each $H_{s}$ is a separable Hilbert space, $s \in X$.

\item[(ii)] $\mathcal{M}$ is (unitarily equivalent to) the direct integral of $\{ \mathcal{M}_{s}: s \in X \}$ over $(X, \mu)$, where $\mathcal{M}_{s}$ is a factor in $B(H_{s})$ almost everywhere. Also, if $\mathcal{M}$ is of type $I_{n}$($n$ could be infinite), II$_{1}$, II$_{\infty}$ or  III, then the components $\mathcal{M}_{s}$ are, almost everywhere, of type I$_{n}$,  II$_{1}$,  II$_{\infty}$ or  III, respectively.
\end{enumerate}

Moreover, the center $\mathcal{Z}_{\mathcal{M}}$ is (unitarily equivalent to) the algebra of diagonalizable operators relative to this decomposition.
\end{lemma}

\begin{lemma} \label{2.5}
(\cite{KR1}) If $H$ is the direct integral of $\{ H_{s} \}$ over $(X, \mu)$, $\mathcal{M}$ is a decomposable von Neumann algebra on $H$ (i.e every operator in $\mathcal{M}$ is decomposable relative to the direct integral decomposition, see Definition 14.1.6 in \cite{KR1}) and $\omega$ is a normal state on $\mathcal{M}$, then there is a mapping, $s \to \omega_{s}$, where $\omega_{s}$ is a positive normal linear functional on $\mathcal{M}_{s}$, and $\omega(a)=\int_{X} \omega_{s}(a(s))d\mu$ for each $a$ in $\mathcal{M}$. If $\mathcal{M}$ contains the algebra $\mathcal{C}$ of diagonalizable operators and $\omega \vert_{E\mathcal{M}E}$ is faithful or tracial, for some projection $E$ in $\mathcal{M}$, then $\omega_{s} \vert_{E(s)\mathcal{M}_{s}E(s)}$ is, accordingly, faithful or tracial almost everywhere.
\end{lemma}

\begin{remark} \label{2.6}
Let $\mathcal{M}=\int_{X} \bigoplus \mathcal{M}_{s} d \mu$ and $H = \int_{X} \bigoplus H_{s} d \mu$ be the direct integral decompositions of $\mathcal{M}$ and $H$ relative to the center of $\mathcal{M}$. By the argument in section 14.1 in \cite{KR1}, we can find a separable Hilbert space $K$ and a family of unitaries $\{U_{s}:H_{s} \to K: s \in X \}$ such that $s \to U_{s}x(s)$ is measurable(i.e. $s \to \langle U_{s}x(s), y \rangle$ is measurable for any vector $y$ in $K$) for every $x\in H$ and $s \to U_{s}a(s)U_{s}^{*}$ is measurable(i.e. $s \to \langle U_{s}a(s)U_{s}^{* }y, z \rangle$ is measurable for any vectors $y,z$ in $K$) for every decomposable operator $a \in B(H)$.
\end{remark}

The following corollary follows directly from  Lemma 14.1.17 and
Lemma 14.1.15 in \cite{KR1}.
\begin{lemma} \label{2.8}
Let $\mathcal{M}$ be a von Neumann algebra acting on a separable
Hilbert space $H$. Let $\mathcal{M}= \int_{X} \bigoplus
\mathcal{M}_{s} d\mu$ and $H=\int_{X} H_{s} d \mu$ be the direct
integral decompositions of $\mathcal{M}$ and $H$ as in Lemma \ref{2.4}. There
exists  a $SOT$ dense sequence $\{ a_{j}: j \in \mathbb{N} \}$ in
the unit ball $(\mathcal{M})_{1}$ of $\mathcal{M}$ (or dense in in
the unit ball $(\mathcal{M}')_{1}$ of $\mathcal{M}$) such that the
sequence $\{ a_{j}(s): j \in \mathbb{N} \}$ is $SOT$ dense in the
unit ball $(\mathcal{M}_{s})_{1}$(or $(\mathcal{M}'_{s})_{1}$,
respectively) for almost every $s \in X$.
\end{lemma}

\section{Type  II$_{1}$ von Neumann algebras with Property $\Gamma$}

Property $\Gamma$ of a type II$_{1}$ factor $\mathcal{A}$ was introduced by Murray and von Neumann in \cite{Mv}. {\em Suppose $\mathcal{A}$ is a type II$_{1}$ factor with a trace $\tau$. Then $\mathcal{A}$ has Property $\Gamma$ if and only if, given $\epsilon>0$ and elements $a_{1}, a_{2}, \dots, a_{k} \in \mathcal{A}$, there exists a unitary $u \in \mathcal{A}$, such that
\begin{enumerate}
\item[(i)] $\tau(u)=0$;

\item[(ii)] $\Vert ua_{j}-a_{j}u \Vert_{2}<\epsilon,  1 \leq j \leq k$.
\end{enumerate}
}

An equivalent definition of Property $\Gamma$ for a type  II$_{1}$ factor was given by Dixmier in \cite{Di1}. {\em Suppose $\mathcal{A}$ is a type  II$_{1}$ factor with a trace $\tau$. It has Property $\Gamma$ if and only if, given $\epsilon>0$, elements $a_{1}, a_{2}, \dots, a_{k} \in \mathcal{A}$ and a positive integer $n$, there exist orthogonal equivalent  projections $\{ p_{i}: 1 \leq i \leq n \} \subset \mathcal{A}$ with sum $I$, such that
$$\Vert p_{i}a_{j} - a_{j}p_{i} \Vert_{2}<\epsilon, 1 \leq i \leq n, 1 \leq j \leq k.$$
}

The following definition of Property $\Gamma$ for a type II$_{1}$ von Neumann algebra was given in \cite{QS}:

\begin{definition} \label{3.1}
(\cite{QS}) Suppose $\mathcal{M}$ is a type II$_{1}$ von Neumann algebra with a
predual $\mathcal M_{\sharp}$. Suppose that $\sigma (\mathcal M,
\mathcal M_\sharp)$ is the weak-$*$  topology on $\mathcal M$ induced
from $\mathcal M_\sharp$. We say that $\mathcal{M}$ has Property
$\Gamma$ if and only if $ \forall \ a_{1}, a_{2}, \dots, a_{k} \in
\mathcal{M}$ and $\forall \ n\in \Bbb N$, there exist a partially ordered set $\Lambda$  and a family of
projections $$\{ p_{i \lambda}: 1\le i\le n;  \lambda \in \Lambda \}\subseteq
\mathcal{M}$$ satisfying
\begin{enumerate}
\item [(i)] For each $\lambda \in \Lambda$,   $\{ p_{1 \lambda}, p_{2 \lambda}, \ldots,
p_{n \lambda} \}$ is a family of orthogonal equivalent projections in
$\mathcal M$ with sum $I$.
\item [(ii)] For each $1\le i\le n$ and $1\le j\le k$,
$$
\lim_{\lambda} (p_{i \lambda}a_{j}-a_{j}p_{i
\lambda})^*(p_{i \lambda}a_{j}-a_{j}p_{i \lambda}) =0 \qquad \text {in $\sigma(\mathcal
M, \mathcal M_\sharp)$ topology.}$$
\end{enumerate}
\end{definition}

It was proved in \cite{QS} that Definition \ref{3.1} coincides with Dixmier's (also with Murray and von Neumann's) definition of Property $\Gamma$ for a type II$_{1}$ factor.

\begin{example}
Let $\mathcal{M}_{1}$ be a type II$_{1}$ factor and $\mathcal{M}_{2}$ a type II$_{1}$ von Neumann algebra. Suppose $\mathcal{M}_{1}$ has Property $\Gamma$ (e.g. hyperfinite type II$_{1}$ factor). Then $\mathcal{M}_{1} \otimes \mathcal{M}_{2}$ has Property $\Gamma$.
\end{example}
Let $\mathcal{M}$ be a type II$_{1}$ von Neumann algebra acting on a separable Hilbert space $H$ and   $\mathcal{Z}_{\mathcal{M}}$ be the center of $\mathcal{M}$. Let $\mathcal{M}=\int_{X} \bigoplus \mathcal{M}_{s} d\mu$ and $H=\int_{X} \bigoplus H_{s} d\mu$ be the direct integral decompositions of $\mathcal{M}$ and $H$ over $(X, \mu)$ relative to $\mathcal{Z}_{\mathcal{M}}$. By Proposition 3.12 in \cite{QS}, $\mathcal{M}$ has Property $\Gamma$ if and only if $\mathcal{M}_{s}$ has Property $\Gamma$ for almost every $s \in X$.

 The following proposition gives an equivalent definition  of Property $\Gamma$ for a type II$_{1}$ von Neumann algebra with separable predual,  as an analogous to Murray and von Neumann's definition for type II$_{1}$ factors.

\begin{proposition} \label{3.5}
Let $\mathcal{M}$ be a type II$_{1}$ von Neumann algebra with separable predual and $\mathcal{Z}_{\mathcal{M}}$ be the center of $\mathcal{M}$. Suppose $\tau$ is a center valued trace from $\mathcal{M}$ to $\mathcal{Z}_{\mathcal{M}}$ such that $\tau(a)=a$  for all $a\in \mathcal {Z}_{\mathcal{M}}$. Then $\mathcal{M}$ has Property $\Gamma$ if and only if there exists a faithful normal tracial state $\rho$ on $\mathcal M$ such that, {\em for any
 $\epsilon >0$ and elements $a_{1}, a_{2}, \dots, a_{k} \in \mathcal{M}$, there exists a unitary $u \in \mathcal{M}$ satisfying (i) $\tau (u)=0$ and (ii) $\Vert ua_{j}-a_{j}u \Vert_{2} < \epsilon$  for all $j=1,2, \dots, k$, where $\Vert \cdot \Vert_{2}$ is the $2$-norm induced by $\rho$.}
\end{proposition}

\begin{proof}
(Part I)  Suppose $\mathcal{M}$ has Property $\Gamma$. Fix $n\ge 2$, $\epsilon >0$ and elements $a_{1}, a_{2}, \dots, a_{k} \in \mathcal{M}$. Then by Corollary 3.4 in \cite{QS},  there exist a faithful normal tracial state $\rho$ and  $n$ equivalent orthogonal projections $p_{1}, p_{2}, \ldots, p_n$ with sum $I$ such that $\Vert p_{i}a_{j}-a_{j}p_{i} \Vert_{2}< \epsilon/n$ for all $  i=1,2, \ldots, n$ and $ j=1,2, \dots, k$, where $\Vert \cdot \Vert_{2}$ is the $2$-norm induced by $\rho$. Let $u=p_{1}+\lambda p_{2}+\cdots +\lambda^{n-1} p_n$, where $\lambda=e^{2\pi i/n}$ is the $n$-th root of unit. Then $u$ is a unitary in $\mathcal{M}$ satisfying $\Vert ua_{j}-a_{j}u \Vert_{2} <\epsilon$ for all  $ j=1,2, \dots, k$. Since $p_{1}$, $p_{2}$, $\ldots, p_n$ are equivalent projections, $\tau (p_{1})=\tau (p_{2})=\cdots =\tau(p_n)$ and thus $\tau (u)=(1+\lambda+\ldots+ \lambda^{n-1})\tau(p_1)=0$.

(Part II) Conversely, suppose that there exists a faithful normal tracial state $\rho$ such that, for any
 $\epsilon >0$ and elements $a_{1}, a_{2}, \dots, a_{k} \in \mathcal{M}$, there exists a unitary $u \in \mathcal{M}$ satisfying (i) $\tau (u)=0$ and (ii) $\Vert ua_{j}-a_{j}u \Vert_{2} < \epsilon$  for all $j=1,2, \dots, k$, where $\Vert \cdot \Vert_{2}$ is the $2$-norm induced by $\rho$. From the fact that  $\mathcal{M}$ has separable predual, it follows $\mathcal M$ is countably generated. Thus, from the preceding argument, we know  there exists a sequence $\{ u_{i}: i \in \mathbb{N} \}$ of unitaries in $\mathcal{M}$ satisfying
\begin{enumerate}
\item[(1)] $\tau (u_{i})=0$ for each $i \in \mathbb{N}$;
\item[(2)] \begin{eqnarray}
\lim\limits_{i \to \infty} \Vert u_{i}a-au_{i}\Vert_{2} = 0 \qquad \text{for each $a \in \mathcal{M}$.} \label{8}
\end{eqnarray}
\end{enumerate}

Since $\mathcal{M}$ has separable predual, by Propostion A.2.1 in \cite{JS}, there is a faithful normal representation $\pi$ of $\mathcal{M}$ on the separable Hilbert space. Replacing $\mathcal{M}$ by $\pi(\mathcal{M})$ in the following if necessary, we assume that $\mathcal{M}$ acts on a separable Hilbert space $H$.

By Lemma \ref{2.4}, relative to $\mathcal{Z}_{\mathcal{M}}$, we obtain a direct integral decomposition $\mathcal{M} = \int_{X} \bigoplus M_{s} d \mu$ over $(X, \mu)$ and we may assume that $\mathcal{M}_{s}$ is a type II$_{1}$ factor with a trace $\tau_{s}$ for every $s \in X$.

Since $\rho$ is a faithful normal tracial state, by Lemma \ref{2.5}, we can further assume that there is a positive faithful normal tracial functional $\rho_{s}$ on $\mathcal{M}_{s}$  for every $s \in X$  such that
$$\rho (a) = \int_{X} \rho_{s} (a(s)) d \mu  \qquad \text{for each $a \in \mathcal{M}$.}.$$
Therefore $\rho_{s}$ is a positive multiple of the trace $\tau_{s}$ on $\mathcal{M}_{s}$ for every $s \in X$.

By Lemma \ref{2.8}, we may assume  $\{ b_{j}: j \in \mathbb{N} \}$
is a $SOT$ dense subset of the unit ball $(\mathcal{M})_{1}$ of
$\mathcal{M}$ such that that $\{ b_{j}(s): j \in \mathbb{N} \}$ is
$SOT$ dense in the unit ball of $\mathcal{M}_{s}$ for every $s \in
X$.

Let $u_{i}^{(0)}=u_{i}$ for each $i \in \mathbb{N}$. In the following we will construct a family of unitaries $\{ u_{i}^{(k)}: i, k \in \mathbb{N} \}$ and  a family of $\mu$-null subsets $\{X_{k}:k\in \Bbb N\}$ of $X$ such that, for each $k \in \mathbb{N}$,
\begin{enumerate}
\item[(i')] $\{ u_{i}^{(k)}: i \in \mathbb{N} \}$ is a subsequence of $\{ u_{i}^{(k-1)}: i \in \mathbb{N} \}$;

\item[(ii')] $\lim\limits_{i \to \infty} \rho_{s} ((u_{i}^{(k)}(s)b_{k}(s)-a_{k}(s)u_{i}^{(k)}(s))^{*}(u_{i}^{(k)}(s)a_{k}(s)-a_{k}(s)u_{i}^{(k)}(s))) = 0$ for any $s \in X \setminus X_{k}$.
\end{enumerate}
By (\ref{8}), we get
$$\begin{aligned}
\lim\limits_{i \to \infty}  &\Vert u_{i}a-au_{i}\Vert_{2}^2  \\ &=\lim\limits_{i \to \infty}\int_{X} \rho_{s} ((u_{i}(s)b_{1}(s)-b_{1}(s)u_{i}(s))^{*}(u_{i}(s)b_{1}(s)-b_{1}(s)u_{i}(s))) d\mu \\ & = 0.\end{aligned}$$
Therefore there exists a $\mu$-null subset $X_{1}$ of $X$ and a subsequence $\{u_{i}^{(1)}\}$ of $\{u_{i}^{(0)}\}$ such that, for any $s \in X\setminus X_{1}$,
\begin{eqnarray}
\lim\limits_{i \to \infty} \rho_{s} ((u_{i}^{(1)}(s)b_{1}(s)-b_{1}(s)u_{i}^{(1)}(s))^{*}(u_{i}^{(1)}(s)b_{1}(s)-b_{1}(s)u_{i}^{(1)}(s))) = 0. \label{9}
\end{eqnarray}
Since $\rho_{s}$ is a positive multiple of $\tau_{s}$ for every $s \in X$, (\ref{9}) gives
\begin{eqnarray}
\lim\limits_{i \to \infty} \Vert u_{i}^{(1)}(s)b_{1}(s)-b_{1}(s)u_{i}^{(1)}(s) \Vert_{2,s}=0, \notag
\end{eqnarray}
where $\Vert \cdot \Vert_{2,s}$ is the $2$-norm induced by $\tau_{s}$ on $\mathcal{M}_{s}$.
Again, there exists a $\mu$-null subset $X_{2}$ of $X$ and a subsequence $\{ u_{i}^{(2)} \}$ of $\{ u_{i}^{(1)} \}$ such that, for any $s \in X \setminus X_{2}$,
\begin{eqnarray}
\lim\limits_{i \to \infty} \Vert u_{i}^{(2)}(s)b_{2}(s)-b_{2}(s)u_{i}^{(2)}(s) \Vert_{2,s} = 0. \notag
\end{eqnarray}
Continuing in this way, we obtain a $\mu$-null subset $X_{k}$ of $X$ and a subsequence $\{u_{i}^{(k)} \}$ of $\{ u_{i}^{(k-1)} \}$  for each $k \geq 1$ such that, for any $s \in X \setminus X_{k}$,
\begin{eqnarray}
\lim\limits_{i \to \infty} \Vert u_{i}^{(k)}(s)b_{k}(s)-b_{k}(s)u_{i}^{(k)}(s) \Vert_{2,s} = 0. \notag
\end{eqnarray}

The argument in the preceding paragraph produces a subsequence $\{u_i^{(i)}\}$ of $\{u_i\}$ such that
$$\lim\limits_{i \to \infty} \Vert u_i^{(i)}(s)b_{j}(s)-b_{j}(s)u_i^{(i)}(s) \Vert_{2,s} = 0$$
for any $j \in \mathbb{N}, s \in X\setminus X_{0}$, where $X_{0}= \cup_{j \in \mathbb{N}} X_{j}$ is a $\mu$-null subset of $X$. Replacing $\{ u_{i}: i \in \mathbb{N} \}$ by $\{u_i^{(i)}\}$ and $X$ by $X \setminus X_{0}$ if necessary,   we might assume
$$\lim\limits_{i \to \infty} \Vert u_{i}(s)b_{j}(s)-b_{j}(s)u_{i}(s) \Vert_{2,s} = 0 \qquad \text{ for any $j \in \mathbb{N}, s \in X$.}$$ Since $\{ b_{j}(s): j \in \mathbb{N} \}$ is SOT dense in the unit ball of $\mathcal{M}_{s}$, we get
\begin{eqnarray}
\lim\limits_{i \to \infty}\Vert u_{i}(s)a-au_{i}(s) \Vert_{2,s} = 0\quad \text { for any $a \in \mathcal{M}_{s}, s \in X$.} \label{12}
\end{eqnarray}

For each $i\ge 1$, since   $u_{i}$ is a unitary in $\mathcal{M}$,   there exists a $\mu$-null subset $Y_{i}$ of X such that $u_{i}(s)$ is a unitary in $\mathcal{M}_{s}$ for each $s \in X \setminus Y_{i}$. Let $Y_{0}= \cup_{i=1}^{\infty} Y_{i}$. Then $\mu (Y_{0})=0$ and $u_{i}(s)$ is a unitary  in $\mathcal M_s$ for all $i \in \mathbb{N}, s \in X\setminus Y_{0}$. So we may just assume that

\begin{equation}\text{$u_{i}(s)$ is a unitary in $\mathcal M_s$ for any $i \in \mathbb{N}, s \in X$.}\label {u1}\end{equation}

For each $i \in \mathbb{N}$, from the Dixmier Approximation Theorem and the fact that  $\tau(u_{i})=0$, $0$ is in the norm closure of the convex hull of $\{ v^{*}u_{i}v$: $v$ is a unitary in $\mathcal{M} \}$. Therefore there exist a sequence of positive integers $\{k_{n} \in \mathbb{N}\}$,  a family of positive numbers $\{ \lambda_{j}^{(n)}: 1 \leq j \leq k_{n}, n\in \Bbb N \} \subseteq [0,1]$ and a family of  unitaries $\{ v_{j}^{(n)}:  1 \leq j \leq k_{n} \}$  in $\mathcal M$ such that
$$ \sum\limits_{j=1}^{k_{n}} \lambda_{j}^{(n)}=1 \quad \text {and} \quad \lim\limits_{n \to \infty} \Vert \sum\limits_{j=1}^{k_{n}} \lambda_{j}^{(n)} (v_{j}^{(n)})^{*}u_{i}v_{j}^{(n)} \Vert = 0.$$
By Proposition 14.1.9 in \cite{KR1}, $\Vert a \Vert$ is the essential bound of $\{ \Vert a(s) \Vert : s \in X \}$ for any $a \in \mathcal{M}$. Note that $\{ v_{j}^{(n)}:  1 \leq j \leq k_{n} \}$ is a family of  unitaries   in $\mathcal M$.   We know that there exists a $\mu$-null subset $Z_{0}$ of X such that
\begin{enumerate}
\item[(a)] for each $n \in \mathbb{N}$, each $1 \leq j \leq k_{n}$ and each $s \in X \setminus Z_{0}$, $v_{j}^{n}(s)$ is a unitary in $\mathcal{M}_{s}$;

\item[(b)] for each $i \in \mathbb{N}$ and each $s \in X \setminus Z_{0}$,
$$\lim\limits_{n \to \infty} \Vert \sum\limits_{j=1}^{k_{n}} \lambda_{j}^{(n)} (v_{j}^{(n)}(s))^{*}u_{i}(s)v_{j}^{(n)}(s) \Vert = 0.$$
\end{enumerate}
Now from (a), (b) and the Dixmier Approximation Theorem, we obtain that
\begin{eqnarray}
\tau_{s}(u_{i}(s))=0, \qquad \text{for any $i \in \mathbb{N}, s \in X \setminus Z_{0}$}. \label{13}
\end{eqnarray}
Here $Z_0$ is a  $\mu$-null subset   of X.

By  (\ref{13}),  (\ref{u1}),  and (\ref{12}), $\mathcal{M}_{s}$ is a type II$_{1}$ factor with Property $\Gamma$ for each $s \in X\setminus Z_{0}$. Therefore by Proposition 3.12 in \cite{QS}, $\mathcal{M}$ has Property $\Gamma$.
\end{proof}

\begin{lemma} \label{mylemma 3.5.1}
Let $\mathcal{M}$ be a countably decomposable type II$_{1}$ von Neumann algebra   and   $\mathcal{Z}_{\mathcal{M}}$ be the center of $\mathcal{M}$. Suppose $\tau$ is a center valued trace from $\mathcal{M}$ to $\mathcal{Z}_{\mathcal{M}}$ such that $\tau(a)=a$  for all $a\in \mathcal {Z}_{\mathcal{M}}$. Suppose there exists a faithful normal tracial state $\rho$ on $\mathcal M$ such that,
\begin{enumerate}\item []
 {\em for any
 $\epsilon >0$ and elements $a_{1}, a_{2}, \dots, a_{k} \in \mathcal{M}$, there exists a unitary $u \in \mathcal{M}$ satisfying (i) $\tau (u)=0$ and (ii) $\Vert ua_{j}-a_{j}u \Vert_{2} < \epsilon$  for all $j=1,2, \dots, k$, where $\Vert \cdot \Vert_{2}$ is the $2$-norm induced by $\rho$.}
\end{enumerate}
Then, for any finite subset $F$ of $\mathcal M$, there exists a  von
Neumann subalgebra $\mathcal M_1$ of $\mathcal M$ such that
\begin{enumerate}
  \item[(1)] $\mathcal M_1$ has separable predual.
\item[(2)] $F\subseteq \mathcal M_1\subseteq \mathcal M$.
  \item[(3)] $\mathcal M_1$ is a type II$_1$ von Neumann algebra with Property $\Gamma$ in the sense of Definition \ref{3.1}.
\end{enumerate}
\end{lemma}

\begin{proof}
  We are going to prove the following claim first.

\vspace{0.1cm} {\bf Claim \ref{mylemma 3.5.1}.1.   }   {\em For any
finite subset $K$ of $\mathcal M$ and any $\epsilon>0$, there exist
an $n\in\Bbb N$,  unitary elements $u, v_1,\ldots, v_n\in\mathcal
M$, such that   $\tau(u)=0$; $ \ \forall \ x\in K$,
$\|ux-xu\|_2<\epsilon$; and
    $\forall x\in K $,  there is an element  $y$ in the convex hull of   $ \{v_1xv_1^*, \ldots, v_{n}xv_{n}^*\} $ satisfying $\|y-\tau(x)\|<\epsilon.$ }

\noindent{Proof of Claim \ref{mylemma 3.5.1}.1: }   From the
assumption on the faithful normal tracial state $\rho$ of $\mathcal
M$, for a finite subset $K$   of $\mathcal M$ and an  $\epsilon>0$,
there is a unitary $u$ in $ \mathcal M$ such that   $\tau (u)=0$ and
  $\Vert ux-xu \Vert_{2} < \epsilon$  for all $x\in K$. By
Dixmier Approximation Theorem, $\tau(x)\in conv_{\mathcal M}(x)^{=}$
for all $x\in\mathcal M$. Therefore, there exist a positive integer
$n$ and  unitary elements $  v_1,\ldots, v_{n}\in\mathcal M$  such
that   for any $x\in K$,  there is an element  $y$ in the convex
hull of   $\{v_1xv_1^*, \ldots, v_{n}xv_{n}^*\} $ satisfying
$\|y-\tau(x)\|<\epsilon.$ This finished the proof of the claim.

\vspace{0.2cm}

\noindent(Continue the proof of Lemma  \ref{mylemma 3.5.1}) \ Let $F$ be a finite subset of $\mathcal M$.

Let $F_1=F$ and $t=1$. From Claim \ref{mylemma 3.5.1}.1, there exist  a positive integer $n_1$,
 unitary elements $u_1, v_1^{(1)},\ldots, v_{n_1}^{(1)}\in\mathcal M$, such that    $\tau(u_1)=0$;    $\|u_1x-xu_1\|_2<1, \  \forall \ x\in F_1 $; and
  for any $x\in F_1$,    there is a   $y$ in the convex hull of   $\{v_1^{(1)}x(v_1^{(1)})^*, \ldots, v_n^{(1)}x(v_n^{(1)})^*\} $ satisfying $\|y-\tau(x)\|<1.$ Let $F_2=F_1\cup\{u_1, v_1^{(1)},\ldots, v_{n_1}^{(1)}\}$.

Assume that $F_1\subseteq F_2\subseteq \ldots \subseteq F_{t}$ have been constructed for some $t\ge 2$. Again, from Claim \ref{mylemma 3.5.1}.1, there exist  a positive integer $n_{t}$,  unitary elements $u_{t}, v_1^{({t})},\ldots, v_{n_{t}}^{({t})}\in\mathcal M$, such that   $\tau(u_{t})=0$;    $\|u_{t}x-xu_{t}\|_2<1/t, \  \forall \ x\in F_t $; and   for any $x\in F_t$,
   there is a   $y$ in the convex hull of   $\{v_{1}^{({t})}x(v_{1}^{({t})})^*, \ldots, v_{n_t}^{({t})}x(v_{n_t}^{({t})})^*\} $ satisfying $\|y-\tau(x)\|<1/t.$ Let $F_{t+1}=F_t\cup\{u_t, v_1^{(t)},\ldots, v_{n_t}^{(t)}\}$.

Continuing the preceding process, we are able to obtain an
increasing sequence $\{F_t\}_{t=1}^\infty$ of finite subsets of $\mathcal
M$ and a sequence of unitaries $\{u_t\}_{t=1}^\infty$ of $\mathcal
M$ such that,  $\forall t\ge 1$,
\begin{enumerate}
\item [(0)]  $ u_t\in F_{t+1}
\subseteq\mathcal M ;  $
\item [(i)] $\tau(u_{t})=0$;
\item [(ii)] $\|u_{t}x-xu_{t}\|_2<1/t, \ \forall \ x\in F_t $;
  \item [(iii)] for $1\le i\le t$,  there is    $y_i$ in the convex hull of   $\{vu_{i} v^*: v \text{ is a unitary element in } F_{t+2} \}
   $ satisfying $\|y_i-\tau(u_i)\|=\|y_i\|<1/(t+1)$;
\end{enumerate}

 Let $\mathcal M_1$ be the von Neumann subalgebra generated by $\{F_t: t\ge 1\}$
 in $\mathcal M$ and $\mathcal Z_1$ be the center of $\mathcal M_1$. Suppose $\tau_1$ be a center-value trace on $\mathcal M_1$ such that
  $\tau_1(a)=a, \ \forall \ a\in \mathcal Z_1$. Then $\rho$ is still a faithful normal tracial state of $\mathcal M_1$. Since $\mathcal M_1$ is countably generated, we know that (1) $\mathcal M_1$ has a separable predual. Obviously, (2) $F\subseteq \mathcal M_1\subseteq \mathcal M$.

We claim  that (3) $\mathcal M_1$  is a von Neumann algebra with Property $\Gamma$ in the
sense of Definition \ref{3.1}.
    Notice $\{F_t\}$ is an increasing sequence of subsets. We have, for each $1\le t_1<t<t_2-2 $,  \begin{enumerate} \item[(ii$_{ 2}$)]
      $\|u_{t}x-xu_{t}\|_2<1/t, \  \forall \ x\in F_{t_1} $;
  \item [(iii$_{2}$)] there is    $y$ in the convex hull of   $ \{vu_{t} v^*: v \text{ is a unitary   in } F_{t_2} \}  $ satisfying
  $\|y\|<1/t_2.$
\end{enumerate}
From (iii$_{2}$), it induces by the Dixmier Approximation Theorem that
  \begin{enumerate}
  \item [(iii$_{3}$)] $\tau_1(u_t)=0$ for each $t\ge 1$.
\end{enumerate}
From the existence of such sequence $\{u_t\}_{t=1}^\infty$ in
$\mathcal M_1$
   satisfying (iii$_{3}$) and (ii$_{2}$), it follows that $\mathcal M_1$
is a type II$_1$ von Neumann algebra.
   From Proposition \ref{3.5}, we conclude that $\mathcal M_1$ has
  Property $\Gamma$.
  The proof of the lemma is finished.
\end{proof}

Now we can quickly prove the following result.

\begin{proposition} \label{3.5.1}
Let $\mathcal{M}$ be a countably decomposable type II$_{1}$ von Neumann algebra  and   $\mathcal{Z}_{\mathcal{M}}$ be the center of $\mathcal{M}$. Suppose $\tau$ is a center valued trace from $\mathcal{M}$ to $\mathcal{Z}_{\mathcal{M}}$ such that $\tau(a)=a$  for all $a\in \mathcal {Z}_{\mathcal{M}}$. Then the following are equivalent.
\begin{enumerate}
\item [(1)] $\mathcal{M}$ has Property $\Gamma$.

 \item [(2)]
 There exist a positive inter $n_0\ge 2$ and a faithful normal tracial state $\rho$ on $\mathcal M$ such that, {\em for any
 $\epsilon >0$ and elements $a_{1}, a_{2}, \dots, a_{k} \in \mathcal{M}$, there exists a family of orthogonal equivalent projections
   $p_1,\ldots, p_{n_0}$ in $\mathcal M$ with sum $I$ satisfying  $\Vert p_ia_{j}-a_{j}p_i \Vert_{2} < \epsilon$  for all $i=1,\ldots, n_0$ and  $j=1,2, \dots, k$, where $\Vert \cdot \Vert_{2}$ is the $2$-norm induced by $\rho$.}
 \item  [(3)]
 There exists a faithful normal tracial state $\rho$ on $\mathcal M$ such that, {\em for any
 $\epsilon >0$ and elements $a_{1}, a_{2}, \dots, a_{k} \in \mathcal{M}$, there exists a unitary $u \in \mathcal{M}$ satisfying (i) $\tau (u)=0$ and (ii) $\Vert ua_{j}-a_{j}u \Vert_{2} < \epsilon$  for all $j=1,2, \dots, k$, where $\Vert \cdot \Vert_{2}$ is the $2$-norm induced by $\rho$.}

 \end{enumerate}
\end{proposition}
\begin{proof}
(1) $\Rightarrow$ (2): It is clear.

 (2) $\Rightarrow$ (3):
 It follows from the similar arguments in  Part I of the proof of Proposition \ref{3.5}.

 (3) $\Rightarrow$ (1):  It follows from Proposition \ref{3.5},  Lemma  \ref{mylemma 3.5.1} and Corollary 3.4 in \cite{QS}.
\end{proof}

The following lemma will be needed in the next section.
\begin{lemma} \label{mylemma 3.5.2} Let $\mathcal M$ be a countably decomposable von Neumann algebra with Property $\Gamma$.
For any finite subset $F$ of $\mathcal M$, there exists a type
II$_1$ von Neumann algebra $\mathcal M_1$ with separable predual and
with Property $\Gamma$ such that $F\subseteq \mathcal M_1\subseteq
\mathcal M$.
\end{lemma}
\begin{proof}
It now follows from Proposition \ref{3.5.1} and Lemma \ref{mylemma 3.5.1}.
\end{proof}

\begin{remark}
It was shown in \cite{QS} that the cohomology group $H^{k}(\mathcal{M}, \mathcal{M})$ of a type II$_{1}$ von Neumann algebra with separable predual and Property $\Gamma$ is trivial for any $k \ge 2$. By a similar argument as section 7 in \cite{EFAR2}, it follows from Theorem 6.4 in \cite{QS} and Lemma \ref{mylemma 3.5.2} that $H^{k}(\mathcal{M}, \mathcal{M})=0$, $k \ge 2$ for any countably decomposable type II$_1$ von Neumann algebra with Property $\Gamma$.
\end{remark}

\section{Similarity degree of   type II$_{1}$ von Neumann algebras with Property $\Gamma$}
Let us recall the  definition of similarity length of a unital C$^*$-algebra as given in \cite{Pi4}.

\begin{definition} \label{4.1}
(\cite{Pi4}) Let $\mathcal{A}$ be a unital C$^*$-algebra. Fix $n \in \mathbb{N}$. For each $x \in M_{n}(\mathcal{A})$ and $d \in \mathbb{N}$, we denote
$$\Vert x \Vert_{(d,\mathcal A)}=\inf\{ \prod\limits_{i=1}^{d} \Vert \alpha_{i} \Vert \prod\limits_{i=0}^{d} \Vert D_{i} \Vert \},$$
where the infimum runs over all possible representations $x=\alpha_{0} D_{1} \alpha_{1} D_{2} \dots D_{d} \alpha_{d}$, $\alpha _{0} \in M_{k,n}(\mathbb{C})$, $\alpha _{1} \in M_{n}(\mathbb{C}), \dots, \alpha _{d-1} \in $ $M_{n}(\mathbb{C}), \alpha _{d} \in M_{n,k}(\mathbb{C})$ are scalar matricies and $D_{1}, D_{2}, \dots, D_{d}  \in M_{n}(\mathcal{A})$ are diagonal matrices.
\end{definition}

It is clear that, for any $x \in M_{n}(\mathcal{A}), d \in \mathbb{N}$,
$$\Vert x \Vert \leq \Vert x \Vert_{(d,\mathcal A)}$$
and
$$\Vert x \Vert_{(d+1,\mathcal A)} \leq \Vert x \Vert_{(d,\mathcal A)}.$$
It was shown in \cite{Pi4} that $\Vert x \Vert_{(1,\mathcal A)} \leq n \Vert x \Vert$ for any $x \in M_{n}(\mathcal{A})$.

Before we prove the main theorem of this section, we need the following lemmas.
\begin{lemma} \label{4.2}
Let $\mathcal{M}$ be a von Neumann algebra acting on a separable Hilbert space $H$. Let $\mathcal{M}= \int_{X} \bigoplus \mathcal{M}_{s} d\mu$ and $H=\int_{X} H_{s} d \mu$ be the direct integral decompositions of $\mathcal{M}$ and $H$ as in Lemma \ref{2.4}. Suppose $p$ is a projection in $\mathcal{M}$ such that there exist a $\mu$-null subset $X_{0}$ of $X$ and projections $\{p_{i,s}\in \mathcal M_s: 2\le i\le n, s\in X\}$   satisfying $\{ p(s), p_{2,s}, \dots, p_{n,s} \}$ is a family of $n$ orthogonal equivalent projections in $\mathcal{M}_{s}$ for all $s\in X\setminus X_0$.  Then there exist  $n-1$ projections $ p_{2}, \dots, p_{n}$ in $\mathcal{M}$ such that $\{ p, p_{2}, \dots, p_{n} \}$ is a set of $n$ orthogonal equivalent projections in $\mathcal{M}$.
\end{lemma}
\begin{proof}
We let $X_{0}$ be a $\mu$-null subset  of $X$ and $\{p_{i,s}\in \mathcal M_s: 2\le i\le n, s\in X\}$  be  projections  satisfying $  p(s), p_{2,s}, \dots, p_{n,s}  $ are  $n$ orthogonal equivalent projections in $\mathcal{M}_{s}$ for all $s\in X\setminus X_0$.

Let $K$ be a separable Hilbert space and $\{ U_{s}: H_{s} \to K \}$ be a family of unitaries as in Remark \ref{2.6}. Denote by $\mathcal{B}$ the unit ball of $B(K)$ equipped with the $*-$strong operator topology. Then $\mathcal{B}$ is metrizable by setting
$$d(S, T)=\sum\limits_{j \in \mathbb{N}} 2^{-j}(\Vert (S-T)e_{j} \Vert + \Vert (S^{*}-T^{*})e_{j} \Vert)$$
for any $S, T \in \mathcal{B}$, where $\{ e_{j}: j \in \mathbb{N} \}$ is an orthonormal basis for $K$. The metric space $(\mathcal{B}, d)$ is complete and separable. For each $i ,j \in \{ 1, 2, \dots, n \}$, let $\mathcal{B}_{ij}= \mathcal{B}$. Let $\mathcal{C}= \prod\limits_{i,j=1}^{n} \mathcal{B}_{ij}$ provided with the product topology of the $*-$strong operator topology on each $\mathcal{B}_{ij}$. It follows that $\mathcal{C}$ is metrizable and it is a complete separable metric space.

By Lemma \ref{2.8}, suppose $\{ a'_{r}: r \in \mathbb{N} \}$ is a
strong operator dense sequence in the unit ball $(\mathcal{M}')_{1}$
of $\mathcal{M}'$ such that that the sequence $\{ a'_{r}(s): r \in
\mathbb{N} \}$ is strong operator dense in the unit ball
$(\mathcal{M}'_{s})_{1}$ of $\mathcal{M}'_{s}$ for every $s \in X$.

We will denote by $(s, (E_{11}, E_{12}, \dots, E_{1n}, E_{21}, \dots, E_{2n}, \dots, E_{nn}))  $ an element in $X \times \mathcal{C}$.
It is easy to see that the maps
\begin{align}
(s, (E_{11}, E_{12}, \dots, E_{1n}, E_{21}, \dots, E_{2n}, \dots, E_{nn})) &\to E_{11}, \label{14} \\
(s, (E_{11}, E_{12}, \dots, E_{1n}, E_{21}, \dots, E_{2n}, \dots, E_{nn})) &\to U_{s}p(s)U_{s}^{*}, \label{15} \\
(s, (E_{11}, E_{12}, \dots, E_{1n}, E_{21}, \dots, E_{2n}, \dots, E_{nn})) &\to E_{ij}, \label{16} \\
(s, (E_{11}, E_{12}, \dots, E_{1n}, E_{21}, \dots, E_{2n}, \dots, E_{nn})) &\to E_{ji}^{*}, \label{17} \\
(s, (E_{11}, E_{12}, \dots, E_{1n}, E_{21}, \dots, E_{2n}, \dots, E_{nn})) &\to E_{ij}E_{kl}, \label{18} \\
(s, (E_{11}, E_{12}, \dots, E_{1n}, E_{21}, \dots, E_{2n}, \dots, E_{nn})) &\to E_{ij}U_{s}a'_{r}(s)U_{s}^{*}, \label{19} \\
(s, (E_{11}, E_{12}, \dots, E_{1n}, E_{21}, \dots, E_{2n}, \dots, E_{nn})) &\to U_{s}a'_{r}(s)U_{s}^{*}E_{ij} \label{20}
\end{align}
are measurable from $X \times \mathcal{C}$ to $\mathcal{B}$ when $\mathcal{C}$ is equipped with the Borel structure obtained from the product topology. By Lemma 14.3.1 in \cite{KR1}, there is a Borel $\mu$-null subset $X_{1}$ of $X$ such that, when restricted to $X\setminus X_{1}$, these maps are all Borel measurable.

Let $N=X_{0} \cup X_{1}$. It follows that $\mu(N)=0$. {\em Let $\eta$ be the set of elements
$$(s, (E_{11}, E_{12}, \dots, E_{1n}, E_{21}, \dots, E_{2n}, \dots, E_{nn})) \in (X\setminus N) \times \mathcal{C}$$
satisfying
\begin{enumerate}
\item[(i)] $E_{11}=U_{s}p(s)U_{s}^{*}$;

\item[(ii)] $E_{ij}=E_{ji}^{*}$ and $E_{ij}E_{kl}=\delta_{jk} E_{il}$ for any $i, j, k, l \in \{ 1, 2, \dots, n\}$;

\item[(iii)] $E_{ij}U_{s}a'_{r}(s)U_{s}^{*} = U_{s}a'_{r}(s)U_{s}^{*}E_{ij}$ for any $i ,j \in \{ 1, 2, \dots, n\}$.
\end{enumerate}
 }

\vspace{0.2cm}

{\bf Claim \ref{4.2}.1:} \ {\em The set $\eta$ is analytic.}

{  Proof of Claim \ref{4.2}.1:} Since the maps (\ref{14})-(\ref{20}) are all Borel measurable when restricted to $X\setminus N$, $\eta$ is a Borel set. By Theorem 14.3.5 in \cite{KR1}, $\eta$ is an analytic set.

The proof of Claim \ref{4.2}.1 is complete.

\vspace{0.2cm}

{\bf Claim \ref{4.2}.2 :} \ {\em Let $\pi$ be the projection of $X \times \mathcal{C}$ onto $X$. Then $\pi(\eta)=X \setminus N$.}

{  Proof of Claim \ref{4.2}.2:} Let $(s, (E_{11}, E_{12}, \dots, E_{1n}, E_{21}, \dots, E_{2n}, \dots, E_{nn}))  $ be an element in $\eta$. Since $\{ a'_{r}(s): r \in \mathbb{N} \}$ is strong operator dense in the unit ball $(\mathcal{M}'_{s})_{1}$ of $\mathcal{M}'_{s}$ for every $s \in X$, conditions   (ii) and  (iii) are equivalent to the statement that $\{U_{s}^{*}E_{ij}U_{s}: 1\le i,j\le n\}  $ is a system of matrix units in $\mathcal M_s$.

 Note $X_0$ is a $\mu$-null subset   of $X$  and $\{p_{i,s}\in \mathcal M_s: 2\le i\le n, s\in X\}$  is a family of projections  satisfying $  p(s), p_{2,s}, \dots, p_{n,s}  $  are $n$ orthogonal equivalent projections in $\mathcal{M}_{s}$ for all $s\in X\setminus X_0$.
Hence for each $s\in X\setminus X_0$, there is a system of matrix units $\{E_{ij}: 1\le i, j\le n\}$ such that (a) $E_{11}=p(s)$ and (b) $E_{ii}=p_{i,s}$ for each $2\le i\le n$. From arguments in the preceding paragraph, it follows that
$$(s, (E_{11}, E_{12}, \dots, E_{1n}, E_{21}, \dots, E_{2n}, \dots, E_{nn}))$$
satisfies conditions (i), (ii) and (iii) for every $s \in X\setminus N$. Therefore $\pi(\eta)=X \setminus N$.

The proof of Claim \ref{4.2}.2 is complete.

\vspace{0.2cm}

\noindent (Continue the proof of Lemma \ref{4.2}) By Claim \ref{4.2}.1 and Claim \ref{4.2}.2, $\eta$ is analytic and $\pi(\eta)= X \setminus N$. It follows from Theorem 14.3.6 in \cite{KR1} that, there is a measurable map
$$s \to (E_{11,s}, E_{12,s}, \dots, E_{1n,s}, E_{21,s}, \dots, E_{2n,s}, \dots, E_{nn,s})$$
from $X\setminus N$ to $\mathcal{C}$ such that
$$(s, (E_{11,s}, E_{12,s}, \dots, E_{1n,s}, E_{21,s}, \dots, E_{2n,s}, \dots, E_{nn,s})) \in \eta$$
for $s \in X\setminus N$. Defining $E_{ij,s}=0$ for $s \in N$ and $ i,j \in \{1, 2, \dots, n\}$, we get {\em a measurable map
$$s \to (E_{11,s}, E_{12,s}, \dots, E_{1n,s}, E_{21,s}, \dots, E_{2n,s}, \dots, E_{nn,s})$$
from $X$ to $\mathcal{C}$ satisfying
\begin{equation}(s, (E_{11,s}, E_{12,s}, \dots, E_{1n,s}, E_{21,s}, \dots, E_{2n,s}, \dots, E_{nn,s})) \in \eta \qquad \text {for almost every $s\in X$}.\label{myequ 1}\end{equation}}

For all $s\in X$, all $i, j \in \{1, 2, \dots, n\}$ and two vectors $x, y \in H$, we have
$$\langle U_{s}^{*} E_{ij,s} U_{s}x(s), y(s) \rangle =\langle E_{ij,s}U_{s}x(s), U_{s}y(s) \rangle.$$
Thus from (\ref{myequ 1}), it follows that the map $s \to \langle U_{s}^{*} E_{ij,s} U_{s}x(s), y(s) \rangle$ is measurable. Since
$$|\langle U_{s}^{*} E_{ij,s} U_{s}x(s), y(s) \rangle| \leq \Vert x(s) \Vert \Vert y(s) \Vert,$$
the map $s \to \langle U_{s}^{*} E_{ij,s} U_{s}x(s), y(s) \rangle$ is integrable. By Definition 14.1.1 in \cite{KR1}, $U_{s}^{*}E_{ij,s}U_{s}x(s) = (p_{ij}x)(s)$ almost everywhere for some $p_{ij}x \in H$. Therefore $p_{ij}(s)=U_{s}^{*} E_{ij,s} U_{s}$ for almost every $s \in X$. It follows from condition (iii) that $p_{ij} \in \mathcal{M}$. For any $i \in \{ 2, 3, \dots, n \}$, let $p_{i}=p_{ii}$. From $p_{ii}(s)=U_{s}^{*}E_{ii,s}U_{s}$, it follows from (i) and (ii) that $p_{2}, p_{3}, \dots, p_{n}$ are the required projections.
\end{proof}

\begin{lemma} \label{4.3}
Let $\mathcal{M}$ be a type  II$_{1}$ von Neumann algebra acting on
a separable Hilbert space $H$. Let $ \mathcal{M}= \int_{X} \bigoplus
\mathcal{M}_{s} d\mu$ and $H=\int_{X} H_{s} d \mu$ be the direct
integral decompositions of $\mathcal{M}$ and $H$ as in Lemma \ref{2.4}.
Suppose that $\mathcal{M}_{s}$ is a type  II$_{1}$ factor with a trace
$\tau_{s}$ for every $s \in X$. Let $n \in \mathbb{N}$ and
$\epsilon$ be a positive number. Let $x=(x_{ij})$ be an element in
$M_{n}(\mathcal{M})$ such that, for every $s \in X$,
$\sum\limits_{i,j=1}^{n}\Vert x_{ij}(s) \Vert_{2,s}^{2} <
\epsilon^{2}$, where the $\Vert \cdot \Vert_{2,s}$ is the $2$-norm
on $\mathcal{M}_{s}$ induced by $\tau_{s}$. Then there are $n$
equivalent orthogonal projections $\{p_{1}, p_{2}, \dots, p_{n}\}$
in $\mathcal{M}$ and $n$ equivalent orthogonal projections ${q_{1},
q_{2}, \dots, q_{n}}$ in $\mathcal{M}$ such that
$$x_{ij}=p_{1}x_{ij}q_{1}+h_{ij}$$
with $\Vert h_{ij} \Vert \leq 3 \epsilon \sqrt{n}$ for every $i, j
\in \{ 1, 2, \dots, n \}$.
\end{lemma}
\begin{proof}
We will use Pisier's trick in \cite{Pi4} to prove the result.

Take $a = (\sum\limits_{i,j=1}^{n} x_{ij}x_{ij}^{*})^{1/2}$. Let
$\mathcal{A}$ be the unital C$^*$-subalgebra of  $\mathcal{M}$
generated by $\{ x_{ij}: 1 \le i, j \le n \}$. By Theorem 14.1.13 in
\cite{KR1} and the fact that $\mathcal{A}$ is a separable
C$^*$-algebra, we know that $\mathcal{A}=\int_{X} \bigoplus \mathcal{A}_{s} d \mu$ and the map $y \to y(s)$ from $\mathcal{A}$ to
$\mathcal{A}_{s}$ is a unital $*$-homomorphism for almost every $s
\in X$.  It follows that $a(s)=(\sum\limits_{i,j=1}^{n}
x_{ij}(s)x_{ij}(s)^{*})^{1/2}$ for almost every $s \in X$. Without
loss of generality, we might assume $a(s)=(\sum\limits_{i,j=1}^{n}
x_{ij}(s)x_{ij}(s)^{*})^{1/2}$ for every $s \in X$. Since
$\sum\limits_{i,j=1}^{n}\Vert x_{ij}(s) \Vert_{2,s}^{2} <
\epsilon^{2}$, we know
\begin{eqnarray}
\Vert a(s) \Vert_{2,s} < \epsilon, \qquad \text{ for every $s \in X$.} \label{21}
\end{eqnarray}

Note that $a$ is a positive element in $\mathcal M$. By functional
calculus, we know there exist an positive integer $k$, a family of
positive numbers $\lambda_1,\ldots, \lambda_k$ and orthogonal
projections $P_1,\ldots, P_k$ in $\mathcal M$ such that
\begin{eqnarray}
\|a-\sum_{i=1}^k \lambda_iP_i\|<\epsilon\qquad \text{ and } \qquad
\|a^2-\sum_{i=1}^k \lambda_i^2P_i\|<\epsilon. \label{myequ 4.3.1}
\end{eqnarray}
Thus $P_1(s),\ldots, P_k(s)$ are orthogonal projections in $\mathcal
M_s$ and
\begin{eqnarray}
\|a(s)-\sum_{i=1}^k \lambda_iP_i(s)\|<\epsilon,
 \label{myequ 4.3.2}
\end{eqnarray}
for  almost every $s\in X.$

From (\ref{21}) and (\ref{myequ 4.3.2}), it follows that
\begin{eqnarray}
\| \sum_{i=1}^k \lambda_iP_i(s)\|_{2,s}<2\epsilon, \qquad \text{for
almost every $s\in X.$}
 \label{myequ 4.3.3}
\end{eqnarray}

\noindent Denote \begin{eqnarray} p_{1}=\chi_{[2\epsilon \sqrt{n},
\infty)}(\sum_{i=1}^k \lambda_iP_i).\label{myequ 4.3.4}
\end{eqnarray}We have
\begin{equation}p_{1}(s)=\left (\chi_{[2\epsilon \sqrt{n},
\infty)}(\sum_{i=1}^k \lambda_iP_i)\right )(s)= \chi_{[2\epsilon
\sqrt{n}, \infty)}(\sum_{i=1}^k \lambda_iP_i(s)), \qquad \text{for
almost every $s\in X.$}\label{myequ 4.3.4}\end{equation}
 By (\ref{myequ 4.3.3}) and (\ref{myequ 4.3.4}),
$$\tau_{s}(p_{1}(s)) < 1/n, \qquad \text{for
almost every $s\in X.$}$$
Since $\mathcal{M}_{s}$ is
a type II$_{1}$ factor for every $s \in X$, there exist  projections $p_{2,s}, p_{3,s},
\dots, p_{n,s}$ in $\mathcal{M}_{s}$ such that $\{p_{1}(s), p_{2,s},
\dots, p_{n,s}\}$ are orthogonal equivalent projections in
$\mathcal{M}_{s}$. By Lemma \ref{4.2}, there exist projections
$p_{2}, p_{3}, \dots, p_{n}$ in $\mathcal{M}$ such that $p_{1},
p_{2}, \dots, p_{n}$ are $n$ orthogonal equivalent projections in
$\mathcal{M}$.

Take $b=(\sum\limits_{i,j=1}^{n} x_{ij}^{*}x_{ij})^{1/2}$. Similarly, we
assume that $ b(s)=(\sum\limits_{i,j=1}^{n}
x_{ij}(s)^{*}x_{ij}(s))^{1/2}$ for every $s \in X$.

Again, note that $b$ is a positive element in $\mathcal M$. By
functional calculus, we know there exist an positive integer $k'$, a
family of positive numbers $\alpha_1,\ldots, \alpha_{k'}$ and
orthogonal projections $Q_1,\ldots, Q_{k'}$ in $\mathcal M$ such that
\begin{eqnarray}
\|b-\sum_{i=1}^{k'} \alpha_iQ_i\|< \epsilon \qquad \text{and} \qquad \Vert b^2-\sum_{i=1}^{k'}\alpha_{i}^{2} Q_i \Vert < \epsilon . \label{myequ 4.3.6}
\end{eqnarray}
Without loss of generality, we can further assume that
\begin{eqnarray}
\|b(s)-\sum_{i=1}^{k'} \alpha_iQ_i(s)\|< \epsilon, \qquad \forall s\in X.
 \label{myequ 4.16}
\end{eqnarray}
By a similar argument as the last paragraph, we obtain a spectral
projection
$$q_{1}= \chi_{[2\epsilon \sqrt{n}, \infty)}(\sum_{i=1}^{k'}
\alpha_iQ_i)$$ and projections $q_{2}, q_{3}, \dots, q_{n}$ such
that $q_{1}, q_{2}, \dots, q_{n}$ are $n$ orthogonal equivalent
projections in $\mathcal{M}$ and $q_{1}(s)=\chi_{[2\epsilon
\sqrt{n}, \infty)}(\sum_{i=1}^{k'} \alpha_iQ_i(s))$ for almost $s \in X$.

Take $h_{ij}=x_{ij}-p_{1}x_{ij}q_{1}$. We may assume that $
h_{ij}(s)=x_{ij}(s)-p_{1}(s)x_{ij}(s)q_{1}(s)$ for every $s \in X$.
In the following we show that $\Vert h_{ij} \Vert \leq 3 \epsilon
\sqrt{n}$. By Proposition 14.1.9 in \cite{KR1}, $\Vert h_{ij} \Vert$
is the essential bound of $\{ \Vert h_{ij}(s) \Vert: s \in X \}$. So
it suffices to show that $\Vert h_{ij}(s) \Vert \leq 3 \epsilon
\sqrt{n}$ for almost every $s \in X$. For every $s \in X$,
\begin{eqnarray}
\Vert h_{ij}(s) \Vert=\Vert x_{ij}(s)-p_{1}(s)x_{ij}(s)q_{1}(s) \Vert \leq \Vert (1-p_{1}(s))x_{ij}(s) \Vert + \Vert p_{1}(s)x_{ij}(s) (1-q_{1}(s)) \Vert. \label{24}
\end{eqnarray}
We have, from (\ref{myequ 4.3.1}) and (\ref{myequ 4.3.6}),
\begin{eqnarray}
&& \Vert (1-p_{1}(s))x_{ij}(s) \Vert^{2}  \notag \\
&=&\Vert (1-p_{1}(s))x_{ij}(s)x_{ij}(s)^{*}(1-p_{1}(s)) \Vert^{2} \notag \\
&\leq&\Vert (1-p_{1}(s))a(s)^{2}(1-p_{1}(s)) \Vert^{2} \notag \\
&\leq&\Vert (1-p_{1}(s))\left(\sum_{i=1}^k \lambda_i^2P_i+
(a(s)^{2}-\sum_{i=1}^k
\lambda_i^2P_i)\right )(1-p_{1}(s)) \Vert^{2} \notag \\
 &\leq& 4\epsilon^{2} n + \epsilon^2\le 5\epsilon^{2} n\label{extra 1}
\end{eqnarray}
and
\begin{eqnarray}
&&\Vert p_{1}(s)x_{ij}(s) (1-q_{1}(s)) \Vert^{2} \notag \\
&\leq& \Vert x_{ij}(s) (1-q_{1}(s)) \Vert^{2} \notag \\
&=& \Vert (1-q_{1}(s))x_{ij}(s)^{*}x_{ij}(s)(1-q_{1}(s))\Vert \notag \\
&\leq& 5\epsilon^{2} n, \label{extra 2}
\end{eqnarray}
for almost every $s\in X$.

It follows from (\ref{24}), (\ref{extra 1}) and (\ref{extra 2})  that for any $i, j \in \{ 1, 2, \dots, n \}$,
$$\Vert h_{ij} \Vert \leq 3 \epsilon \sqrt{n}.$$
The proof is complete.
\end{proof}

Now we are ready to prove the following result as a generalization
of Theorem 13 in \cite{Pi4}.

\begin{theorem} \label{mainthm1}
Let $\mathcal{M}$ be a countably decomposable type II$_{1}$ von Neumann algebra. If $\mathcal M$ has Property $\Gamma$, then
$d(\mathcal{M}) \leq 5$.
\end{theorem}
\begin{proof}

 To show that $d(\mathcal{M})=l(\mathcal M) \leq 5$, it suffices to prove that there
exists a positive scalar $k$ such that for any $n \in \mathbb{N}$
and any $x=(x_{ij}) \in M_{n}(\mathcal{M})$,
$$\Vert x \Vert_{(5,\mathcal M)} \leq k \Vert x \Vert.$$

In the following we fix $n \in \mathbb{N}$ and $x=(x_{ij}) \in M_{n}(\mathcal{M})$.  We may assume that $\Vert x \Vert=1$.

Let $F=\{x_{ij}\}_{i,j=1}^n$ be a finite subset of $\mathcal M$. Note  $\mathcal M$ is a countably decomposable type II$_{1}$ von Neumann algebra with  Property $\Gamma$. From Lemma \ref{mylemma 3.5.2}, it follows that there exists a von Neumann algebra $\mathcal M_1$ with separable predual and with Property $\Gamma$ such that $F\subseteq \mathcal M_1\subseteq \mathcal M$. It is easy to see from Definition \ref{4.1} that
$$\Vert x \Vert_{(5,\mathcal M)} \leq \Vert x \Vert_{(5,\mathcal M_1)}.$$
Hence, to prove the theorem, it suffices to show that there exists a universal constant $k>0$ such that
$$\Vert x \Vert_{(5,\mathcal M_1)} \leq k .$$
Replacing $\mathcal M$ by $\mathcal M_1$, we can assume that $\mathcal{M}$ has separable predual and  Property $\Gamma$.

Since $\mathcal{M}$ has separable predual, by Proposition A.2.1 in
\cite{JS}, there is a faithful normal representation $ \pi$ of
$\mathcal{M}$ on a separable Hilbert space $H$. Replacing
$\mathcal{M}$ by $\pi(\mathcal{M})$ if necessary, we assume that
$\mathcal{M}$ is acting on a separable Hilbert space.

Let $\mathcal{Z}_{\mathcal{M}}$ be the center of $\mathcal{M}$. Let
$ \mathcal{M}= \int_{X} \bigoplus \mathcal{M}_{s} d \mu$ and $H=
\int_{X} \bigoplus H_{s} d\mu$ be the direct integral decompositions
of $\mathcal{M}$ and $H$ relative to $\mathcal{Z}_{\mathcal{M}}$ as in
Lemma \ref{2.4}. By Proposition 3.12 in \cite{QS}, we may assume
that $\mathcal{M}_{s}$ is a type II$_{1}$ factor with Property $\Gamma$ for every $s \in
X$. For any $\epsilon>0$, applying Corollary 4.2 in \cite{QS}, we
obtain $n$ orthogonal equivalent projections $p_{1}, p_{2}, \dots,
p_{n}$ in $\mathcal{M}$ summing to $I$ such that
\begin{eqnarray}
\sum\limits_{i, j=1}^{n} \Vert (x_{ij}-\sum\limits_{m=1}^{n}
p_{m}x_{ij}p_{m})(s) \Vert_{2,s}^{2} < \epsilon^{2}, \label{25}
\end{eqnarray}
where $\Vert \cdot \Vert_{2,s}$ is the $2$-norm induced by the
unique trace $\tau_{s}$ on $\mathcal{M}_{s}$ for   $s \in X$ almost
everywhere.

Decompose
\begin{eqnarray}
x=(\sum\limits_{m=1}^{n} p_{m}x_{ij}p_{m})+ (x_{ij}-\sum\limits_{m=1}^{n} p_{m}x_{ij}p_{m}). \label{26}
\end{eqnarray}

Let $x_{m}=(p_{m}x_{ij}p_{m})$ for $1\le m\le n$. Therefore $x_{m}=
(1 \otimes p_{m}) x (1 \otimes p_{m})$ for $1\le m\le n$. By
applying Lemma 5 in \cite{Pi4}, for each $m$,
$$\Vert x_{m} \Vert_{(3,\mathcal M)} \leq 1.$$
Since $(\sum\limits_{m=1}^{n}
p_{m}x_{ij}p_{m})=(\sum\limits_{m=1}^{n} p_{m} (p_{m}x_{ij}p_{m})
p_{m})$, by Lemma 14 in \cite{Pi4},
\begin{eqnarray}
\Vert (\sum\limits_{m=1}^{n} p_{m}x_{ij}p_{m}) \Vert_{(5,\mathcal
M)} \leq 1. \label{extra 3}
\end{eqnarray}

Let $x'_{ij}=x_{ij}-\sum\limits_{m=1}^{n} p_{m}x_{ij}p_{m}$ for each
$i, j \in \{ 1, 2, \dots, n \}$. It follows that $\Vert (x'_{ij})
\Vert \leq 2$. By Lemma \ref{4.2}, (\ref{25}) implies that there exist $n$
orthogonal equivalent projections $p'_{1}, p'_{2}, \dots, p'_{n}$
and $n$ orthogonal equivalent projections $q'_{1}, q'_{2}, \dots,
q'_{n}$ such that
$$x'_{ij}=p'_{1}x'_{ij}q'_{1}+ h'_{ij}$$
with $\Vert h'_{ij} \Vert \leq 3\epsilon \sqrt{n}$. By Lemma 5 in
\cite{Pi4},
\begin{eqnarray}
\Vert (p'_{1}x'_{ij}q'_{1}) \Vert_{(5,\mathcal M)} \leq \Vert
(p'_{1}x'_{ij}q'_{1}) \Vert_{(3,\mathcal M)} =\Vert (1 \otimes
p'_{1}) (x'_{ij}) (1 \otimes q'_{1})  \Vert_{(3,\mathcal M)} \leq 2.
\label{extra 4}
\end{eqnarray}

Since $\Vert h'_{ij} \Vert \leq 3\epsilon \sqrt{n}$ for every $i, j
\in \{ 1, 2, \dots, n \}$, we obtain
\begin{eqnarray}
\Vert (h'_{ij}) \Vert_{(5,\mathcal M)} &\leq& n \Vert (h'_{ij}) \Vert \notag \\
&\leq& n^{3} \sup \{ \Vert h'_{ij} \Vert: i, j \in \{ 1, 2, \dots, n\} \} \notag \\
&\leq& 3n^{3} \sqrt{n} \epsilon. \label{extra 5}
 \end{eqnarray}

By (\ref{26}), (\ref{extra 3}), (\ref{extra 4}) and (\ref{extra 5}),
$\Vert x \Vert_{(5,\mathcal M)} \leq 3+ 3n^{3} \sqrt{n} \epsilon$.
Since $\epsilon$ was arbitrarily chosen, we obtain that  $\Vert x
\Vert_{(5,\mathcal M)} \leq 3$. Therefore $$d(\mathcal{M}) \leq 5.$$
The proof is complete.
\end{proof}

\section{Similarity degree of a class of C$^*$-algebras}
In this section we will deal with the similarity degree of a class
of C$^*$-algebras with certain properties. We will conclude
that $d(\mathcal{A}) \leq 3$ if $\mathcal{A}$ is a separable unital
$\mathcal{Z}$-stable C$^*$-algebra.

Before we prove the main theorem of this section, we give the following lemma.

\begin{lemma} \label{5.1}
Let $\mathcal{M}$ be a von Neumann algebra with the type decomposition
$$\mathcal{M}= \mathcal{M}_{1} \oplus \mathcal{M}_{c_{1}} \oplus \mathcal{M}_{c_{\infty}} \oplus \mathcal{M}_{\infty},$$
where $\mathcal{M}_{1}$ is a type  I  von Neumann algebra,
$\mathcal{M}_{c_{1}}$ is a type  II$_{1}$ von Neumann algebra,
$\mathcal{M}_{c_{\infty}}$ is a type II$_{\infty}$ von Neumann
algebra and $\mathcal{M}_{\infty}$ is a type  III  von Neumann
algebra. Suppose  $\mathcal{M}_{c_{1}}$ has finite similarity degree. If $\phi$ is a bounded
unital representation of $\mathcal M$ on a Hilbert space $H$, which
is continuous from $\mathcal M$, with the topology $\sigma(\mathcal
M,\mathcal M_\sharp)$, to $B(H),$ with the topology
$\sigma(B(H),B(H)_\sharp)$, then $\phi$ is completely bounded.
\end{lemma}

\begin{proof}
Observe that there is no trace on $\mathcal{M}_{c_\infty} \oplus \mathcal{M}_{\infty}$. It follows from Theorem 8 in \cite{Pi4} that $\mathcal{M}_{c_\infty} \oplus \mathcal{M}_{\infty}$ has finite similarity degree. From the hypothesis, $\mathcal{M}_{c_{1}}$ has finite similarity degree.

Since $\mathcal{M}_{1}$ is of type I, there is a directed collection
$\{ \mathcal{A}_{i}: i\in I \}$ of finite dimensional unital
C$^*$-subalgebras such that $\cup_{i \in I} \mathcal{A}_{i}$ is
 dense in $\mathcal{M}_{1}$ under the topology $\sigma(\mathcal{M},
\mathcal{M}_{\sharp})$. Let $\mathcal{A}$ be the norm closure of
$\cup_{i \in I} \mathcal{A}_{i}$. Then $\mathcal{A}$ is also dense
in $\mathcal{M}_{1}$ under the topology $\sigma(\mathcal{M},
\mathcal{M}_{\sharp})$. Each $\mathcal{A}_{i}$ is finite
dimensional, therefore $\mathcal{A}_{i}$ is nuclear for each $i \in
I$. It follows from Proposition 11.3.12 in \cite{KR1} that
$\mathcal{A}$ is nuclear. By Theorem 4.1 in \cite{C3}, $\mathcal{A}$
has finite similarity degree.

Now let $\mathcal{B} = \mathcal{A} \oplus \mathcal{M}_{c_{1}} \oplus
\mathcal{M}_{c_\infty} \oplus \mathcal{M}_{\infty}$. Then
$\mathcal{B}$ is dense in $\mathcal{M}$ under the topology
$\sigma(\mathcal M,\mathcal M_\sharp)$ and we can obtain that
$\mathcal{B}$ has finite similarity degree by Remark 6 in
\cite{Pi1}. Thus the restriction of $\phi$ to $\mathcal{B}$ is
completely bounded. Notice that $\phi$ is   continuous from
$\mathcal M$, with the topology $\sigma(\mathcal M,\mathcal
M_\sharp)$, to $B(H),$ with the topology $\sigma(B(H),B(H)_\sharp)$.
Now we conclude that $\phi$ is a completely bounded representation
of $\mathcal M$ on $H$.

\end{proof}

The following Theorem gives an estimation of $\Vert \phi \Vert_{cb}$ in the case that $\mathcal{M}_{c_{1}}$ is a countably decomposable type II$_1$ von Neumann algebra with Property $\Gamma$. The main idea  of the proof of the following theorem is based on the one used by Christensen   in \cite{C2}.

\begin{theorem} \label{mainthm2}
Let $\mathcal{M}$ be a von Neumann algebra with the type decomposition
$$\mathcal{M}= \mathcal{M}_{1} \oplus \mathcal{M}_{c_{1}} \oplus \mathcal{M}_{c_{\infty}} \oplus \mathcal{M}_{\infty},$$
where $\mathcal{M}_{1}$ is a type  I  von Neumann algebra,
$\mathcal{M}_{c_{1}}$ is a type  II$_{1}$ von Neumann algebra,
$\mathcal{M}_{c_{\infty}}$ is a type II$_{\infty}$ von Neumann
algebra and $\mathcal{M}_{\infty}$ is a type  III  von Neumann
algebra. Suppose  $\mathcal{M}_{c_{1}}$ is a countably decomposable
von Neumann algebra with  Property $\Gamma$. If $\phi$ is a bounded
unital representation of $\mathcal M$ on a Hilbert space $H$, which
is continuous from $\mathcal M$, with the topology $\sigma(\mathcal
M,\mathcal M_\sharp)$, to $B(H),$ with the topology
$\sigma(B(H),B(H)_\sharp)$, then $\phi$ is completely bounded and
$\Vert \phi \Vert_{cb} \leq \Vert \phi \Vert^{3}$.
\end{theorem}
\begin{proof}
By Theorem \ref{mainthm1} we can obtain that $\mathcal{M}_{c_{1}}$ has finite similarity degree. It follows directly from Lemma \ref{5.1} that $\phi$ is completely bounded. Therefore there exists an invertible positive operator $t \in B(H)$ such that $\pi(\cdot)=t \phi(\cdot)
t^{-1}$ is a unital $*$-homomorphism.  Since $\phi$ is continuous
from $\mathcal M$ in $\sigma(\mathcal M,\mathcal M_\sharp)$ topology
to $B(H)$ in $\sigma(B(H),B(H)_\sharp)$ topology, $\pi$ is
continuous from $\mathcal M$ in $\sigma(\mathcal M,\mathcal
M_\sharp)$ topology to $B(H)$ in $\sigma(B(H),B(H)_\sharp)$
topology.

Let $\mathcal I=ker(\pi)$. Then $\mathcal I$ is a
 two-sided idea in $\mathcal M$ and is closed in  $\sigma(\mathcal M,\mathcal M_\sharp)$
 topology. From Proposition 1.10.5 in \cite{Sa}, it follows that
 there is a central projection $p$ in $\mathcal M$
 such that $\mathcal I=p\mathcal M$. Let $\phi_p, \pi_{p}: \mathcal (I-p)\mathcal M\rightarrow B(H)$ be defined by$$ \phi_{p}(a)=\phi( a)\qquad \text{and} \qquad \pi_p(a)=\pi( a) \quad \text{ for $a\in (I-p)\mathcal M$.}$$
 By the choice of $p$, we know that $\phi_p$ and $\pi_p$ are injective homomorphisms from $(I-p)\mathcal M$ to $B(H)$ such that
 $\|\phi_p\|=\|\phi\|,$ $\|\phi_p\|_{cb}=\|\phi\|_{cb},$ and
 $\phi_p(\cdot)=t\pi_p(\cdot)t^{-1}$. Moreover, by the fact that $p$ is a central projection, we have
 $$
 (I-p)\mathcal{M}= (I-p)\mathcal{M}_{1} \oplus (I-p)\mathcal{M}_{c_{1}} \oplus (I-p)\mathcal{M}_{c_{\infty}} \oplus (I-p)\mathcal{M}_{\infty},
 $$ and $(I-p)\mathcal M_{c_1}$ is a type II$_1$ von Neumann algebra with Property $\Gamma$.
 Replacing $\mathcal M, \phi,$
 and $
 \pi$ by $(I-p)\mathcal M, \phi_p,$ and $
 \pi_p$ if necessary, we might assume that
 $$\text{$\phi$  is an
 injective homomorphism and  $\pi$ is a $*$-isomorphism.}$$
 Dividing $t$
  by $\Vert t\Vert$ if necessary, we assume that $\Vert t \Vert =
  1$.

Let $(\mathcal{M})_{1}$ be the unit ball of $\mathcal{M}$ and
$(\pi(\mathcal{M}))_{1} $ be the unit ball of $\pi(\mathcal{M})$.
Note $\sigma(\mathcal M,\mathcal M_\sharp)$ topology (or
$\sigma(B(H),B(H)_\sharp)$ topology) coincides with the weak
operator topology ($WOT$) on bounded subsets of $\mathcal{M}$ (or
$B(H)$ respectively). Since $\pi$ is continuous from $\mathcal M$ in
$\sigma(\mathcal M,\mathcal M_\sharp)$ topology to $B(H)$ in
$\sigma(B(H),B(H)_\sharp)$ topology,  $\pi$ is $WOT$-$WOT$
continuous when restricted to bounded subsets of $\mathcal{M}$.
Therefore $\pi(\mathcal{M})$ is a von Neumann algebra.

Let $p_{1}, p_{2}, p_{3}$ be central projections in $\mathcal{M}$
with sum $I$ such that $ p_{1}\mathcal{M}=\mathcal{M}_{1}$, $p_{2}
\mathcal{M}=\mathcal{M}_{c_{1}}$ and
$p_{3}\mathcal{M}=\mathcal{M}_{c_{\infty}} \oplus
\mathcal{M}_{\infty}$. By assumption on $\mathcal M$, we know that $\mathcal{M}_{c_{1}}$ is a type II$_1$ von Neumann algebra with Property $\Gamma$. Moreover $$q_{i}=\pi(p_{i})$$ is a central
projection in $\pi(\mathcal{M})$ for each $i =1, 2, 3$ and we can
decompose
$$\pi=\pi_{1} \oplus \pi_{2} \oplus  \pi_{3},$$
where$$\text{ $\pi_{i}(a)=\pi(p_{i}a)=q_{i}\pi(a)$,  \ \ $\forall a
\in \mathcal{M}$, \ $i=1, 2, 3$.}$$ Since $\pi$ is $WOT$-$WOT$
continuous when restricted to bounded subsets of $\mathcal{M}$, each
$\pi_{i}$ is $WOT$-$WOT$ continuous on bounded sets of
$\mathcal{M}$. If one of $\pi_{1}, \pi_{2}, \pi_{3}$ is trivial, the
following proof will be simplified. Here we assume that they are all
nontrivial.  By a similar argument as the preceding paragraph, we
obtain that  $\pi_{i}(\mathcal{M})$ is a von Neumann algebra over
$q_{i}H$ for $i=1,2,3$. Since $\pi$ is a $*$-isomorphism from
$\mathcal M_{c_1}$ onto $\pi(\mathcal M_{c_1})=\pi_2(\mathcal
M_{c_1})$ and $\mathcal M_{c_1}$ is  a countably decomposable  type
II$_1$ von Neumann algebra with Property $\Gamma$, $\pi_2(\mathcal
M_{c_1})$ is a countably decomposable  type II$_1$ von Neumann
algebra with Property $\Gamma$.

Let $c=\Vert \phi \Vert$. It follows that for any $a \in (\mathcal{M})_{1}$,
$$\phi(a)\phi(a)^{*} \leq c^{2}.$$
Since $\pi(\cdot)=t \phi(\cdot)t^{-1}$ and $t$ is positive, we
obtain that for any $a \in (\mathcal{M})_{1}$,
\begin{eqnarray}
\pi(a)t^{2}\pi(a)^{*} \leq c^{2} t^{2}. \label{31}
\end{eqnarray}
Since $\pi$ is a $*$-isomorphism between $\mathcal M$ and
$\pi(\mathcal M)$, we know that $\pi((\mathcal{M})_{1})=
(\pi(\mathcal{M}))_{1}$. Now, equation (\ref{31}) implies that for
any unitary $u \in \pi(\mathcal{M})$, $ut^{2}u^{*} \leq
c^{2}t^{2}$. Therefore for any unitaries $u, v \in
\pi(\mathcal{M})$,
\begin{eqnarray}
ut^{2}u^{*} \leq c^{2} vt^{2}v^{*}. \label{32}
\end{eqnarray}

To show $\Vert \phi \Vert_{cb} \leq c^{3}$, we just need to prove that for any $n \in \mathbb{N}$ and any unitary $U \in M_{n}(\mathcal{M})$,
\begin{eqnarray}
\pi^{(n)}(U) (I_{n} \otimes t^{2}) \pi^{(n)}(U^{*}) \leq c^{6} (I_{n} \otimes t^{2}), \label{33}
\end{eqnarray}
where $\pi^{(n)}: M_{n}(\mathcal{M}) \to B(H^{n})$ is the $n-$folded map of $\pi$.

Now fix $n \in \mathbb{N}$, $\epsilon >0$ and a unitary $U \in
M_{n}(\mathcal{M})$. Let $ x=(x_{1}, x_{2}, \dots, x_{n})$ be a unit
vector in $H^{n}$.  Denote $$\text{$x^{(i)}=(I_{n} \otimes
p_{i})x=(x_{1}^{(i)}, x_{2}^{(i)}, \dots, x_{n}^{(i)})$, \qquad  $i=1, 2,
3$.} $$ In the following we will carry out computations according to
the representations $\pi_{1}, \pi_{2}, \pi_{3}$ and prove inequality
(\ref{33}) in each case.

We will first replace $t^{2}$ by
an element that commutes with $q_{1}, q_{2}$ and $q_{3}$.

 Note $\pi_{2}(\mathcal{M}_{c_{1}})$ is a countably decomposable  type II$_1$ von Neumann algebra with Property $\Gamma$. Thus there exists  a faithful normal tracial state $\rho$ on $\pi_{2}(\mathcal{M}_{c_{1}})$. Let $\|\cdot\|_2$ be the $2$-norm induced by $\rho$ on $\pi_2(\mathcal{M}_{c_{1}})$. Then $2$-norm topology  coincides with the strong operator topology on bounded subsets of $\pi_{2}(\mathcal{M}_{c_{1}})$.
 It follows from Corollary 3.4 in \cite{QS} that we can obtain $n$ orthogonal equivalent projections $r_{1}, r_{2}, \dots, r_{n}$ in $\pi_{2}(\mathcal{M}_{c_{1}})$  with sum $q_{2}$ such that, for each $i \in \{ 1, 2, \dots, n\}$,
\begin{eqnarray}
\Vert \left (I_{n} \otimes ( 1-r_{i} )\right ) (\pi_{2})^{(n)}(U_{2})\left (I_{n} \otimes   r_{i}  \right) x^{(2)} \Vert \leq  (\frac{\epsilon}{n^{3}})^{1/2}. \label{34}
\end{eqnarray}
From  the fact that projections $r_{1}, r_{2}, \dots, r_{n}$ are orthogonal equivalent projections   in $\pi_{2}(\mathcal{M}_{c_{1}})$,  there exists a system of matrix units $\{ r_{ij}: i, j=1, 2, \dots, n \}\subseteq\pi_{2}(\mathcal{M}_{c_{1}})$ in such that $r_{ii}=r_{i}$ for each $i$. Denote by $\mathcal{N}$ the von Neumann algebra generated by $\{ r_{ij}: i, j=1, 2, \dots, n \}$. It follows that $\mathcal{N} \cong M_{n}(\mathbb{C})$ and $q_{2} \in \mathcal{N}$.

Since $\mathcal{M}_{c_{\infty}}$ is of type  II$_{\infty}$ and $\mathcal{M}_{\infty}$ is of type  III, there exist subfactors $\mathcal B_2$ in $\mathcal{M}_{c_{\infty}}$ and $\mathcal B_3$ in $\mathcal{M}_{\infty}$ such that $\mathcal B_2\cong \mathcal B_3\cong B(l^2(\Bbb N)$. Hence,  there exists a subfactor $\mathcal{B}$ of $\pi_{3}(\mathcal{M})$ (=$\pi_{3}(\mathcal{M}_{c_{\infty}} \oplus \mathcal{M}_{\infty})$) such that $\mathcal{B} \cong B(l^{2}(\mathbb{N}))$ and $q_{3} \in \mathcal{B}$. Take $\mathcal{K}=\mathcal{B} ' \cap \pi_{3}(\mathcal{M})$. By Lemma 11.4.11 in \cite{KR1},
$$\pi_{3}(\mathcal{M}) \cong \mathcal{B} \otimes \mathcal{K}.$$

Let $\mathcal{R}=\pi_{1}(\mathcal{M}) \oplus \mathcal{N} \oplus \mathcal{B}$. Since $\mathcal{M}_{1}$ is of type  I, the von Neumann algebra $\pi_{1}(\mathcal{M})$ (=$\pi_{1}(\mathcal{M}_{1})$) is of type  I. By the choices of $\mathcal B$ and $\mathcal N$, we know that $\mathcal{R}$ is a type I von Neumann algebra  and thus   an injective von Neumann algebra. Let $\mathcal{U}$ be the unitary group of $\mathcal{R}$. Hence the set
\begin{eqnarray}
W=\overline{conv}^{uw} \{ ut^{2}u^{*}: u \in \mathcal{U} \} \cap \mathcal{R}'  \label{35}
\end{eqnarray}
is nonempty, where $\overline{conv}^{uw} \{ ut^{2}u^{*}: u \in \mathcal{U} \}$ denotes the  closure of the convex hull of the set $\{ ut^{2}u^{*}: u \in \mathcal{U} \}$ in the ultraweak topology. Choose $m \in W$. It is clear to see that $\Vert m \Vert \leq \Vert t^{2} \Vert \leq 1$. By equation (\ref{32}), for any unitaries $u, v \in \pi(\mathcal{M})$,
\begin{eqnarray}
umu^{*} \leq c^{2} vmv^{*}. \label{36}
\end{eqnarray}
Since $q_1,q_2 $ and $q_3$ are in $\mathcal R$, we obtain that $m$ commutes with $q_{1}, q_{2}$ and $q_{3}$. Let
$$m_{1}=q_{1}m, \qquad  m_{2}=q_{2}m \qquad \text { and } \qquad m_{3}=q_{3}m.$$
 Therefore inequality (\ref{36}) implies
\begin{enumerate}
\item[(i)] for any unitaries $u , v \in \pi_{1}(\mathcal{M}_{1})$,
\begin{eqnarray}
um_{1}u^{*} \leq c^{2}vm_{1}v^{*}; \label{37}
\end{eqnarray}

\item[(ii)] for any unitaries $u, v \in \pi_{2}(\mathcal{M}_{c_{1}})$,
\begin{eqnarray}
um_{2}u^{*} \leq c^{2}vm_{2}v^{*}; \label{38}
\end{eqnarray}

\item[(iii)] for any unitaries $u, v \in \pi_{3}(\mathcal{M}_{c_{\infty}} \oplus \mathcal{M}_{\infty})$,
\begin{eqnarray}
um_{3}u^{*} \leq c^{2}vm_{3}v^{*}. \label{39}
\end{eqnarray}
\end{enumerate}

Since $m \in \mathcal{R}'$, it commutes with $\pi_{1}(\mathcal{M})$. Thus $I_{n}\otimes m_{1}$ commutes with $\pi_{1}^{(n)}(\mathcal{M})$. So we obtain that
\begin{eqnarray}
(\pi_{1})^{(n)}(U)(I_{n}\otimes m_{1}) (\pi_{1})^{(n)}(U^{*}) =(I_{n}\otimes m_{1})(\pi_{1})^{(n)}(U)(\pi_{1})^{(n)}(U^{*})= I_{n} \otimes m_{1}. \label{40}
\end{eqnarray}

Since $m \in \mathcal{R}'$ and $\mathcal{K}=\mathcal{B}' \cap \pi_{3}(\mathcal{M})$, we get $m_{3} \in \mathcal{K}$. Combining inequality  (\ref{39}) with the facts that  $\mathcal{B} \cong B(l^{2}(\mathbb{N})),   \pi_{3}(\mathcal{M}) \cong \mathcal{B} \otimes \mathcal{K}$ and $m\in \mathcal K$, we conclude that
\begin{eqnarray}
(\pi_{3})^{(n)}(U)(I_{n} \otimes m_{3})(\pi_{3})^{(n)}(U^{*}) \leq c^{2}(I_{n} \otimes m_{3}). \label{41}
\end{eqnarray}

In the following we will deal with   $(\pi_{2})^{(n)}$ and $\mathcal M_{c_1}$. We will use  arguments  similar to the ones used in the proof of Lemma 5 in \cite{Pi4} and in the proof of Theorem 2.3 in \cite{C2}.

Recall  $r_{1}, r_{2}, \dots, r_{n}$  are orthogonal equivalent projections in $\pi_{2}(\mathcal{M}_{c_{1}})$  such that
\begin{equation} r_{1}+ r_{2}+ \cdots+ r_{n} =q_2 \label{41.5}.\end{equation} Also recall that $\{ r_{ij}: i, j=1, 2, \dots, n \}$ is a system of matrix units of $\mathcal N$ satisfying, for each $i \in \{ 1, 2, \dots, n\}$, we have $r_{ii}=r_{i}$   and
\begin{eqnarray}
\Vert \left (I_{n} \otimes ( 1-r_{i} )\right ) (\pi_{2})^{(n)}(U_{2})\left (I_{n} \otimes   r_{i}  \right) x^{(2)} \Vert \leq  (\frac{\epsilon}{n^{3}})^{1/2}.   \label{39.1}
\end{eqnarray}
Let $(e_{ij})$ be a system of matrix units for $M_{n}(\mathbb{C})$. For $k=1, 2, \dots, n$, let 
$$f_{k}= \sum\limits_{i=1}^{n} e_{i1} \otimes r_{ki}.$$
Therefore for each $k=1,\ldots, n$,\begin{equation}\text{ $f_{k}f_{k}^{*}= I_{n} \otimes r_{k}$, \ \  $f_{k}^{*}f_{k}=e_{11} \otimes q_{2}$ \ and  \ $f_{k} (I_{n} \otimes m_{2})=(I_{n} \otimes m_{2})f_{k}$.} \label{39.0}\end{equation} Since $\Vert m \Vert \leq 1$, by equation  (\ref{41.5}) and inequalities (\ref{39.1}), (\ref{39.0}) we obtain
\begin{eqnarray}
&&\langle (\pi_{2})^{(n)}(U) (I_{n} \otimes m_{2}) (\pi_{2})^{(n)}(U^{*}) x^{(2)}, x^{(2)} \rangle  \notag \\
&=& \sum\limits_{i, j, k=1}^{n} \langle (I_{n} \otimes r_{i}) (\pi_{2})^{(n)}(U) (I_{n} \otimes r_{j}) (I_{n} \otimes m_{2}) (\pi_{2})^{(n)} (U^{*})(I_{n} \otimes r_{k}) x^{(2)}, x^{(2)}   \notag \rangle \\
&\leq& \sum\limits_{k=1}^{n} \langle (I_{n} \otimes r_{k}) (\pi_{2})^{(n)}(U) (I_{n} \otimes r_{k}) (I_{n} \otimes m_{2}) (\pi_{2})^{(n)} (U^{*})(I_{n} \otimes r_{k}) x^{(2)}, x^{(2)} \rangle + \epsilon \notag \\
&=&\sum\limits_{k=1}^{n} \langle f_{k}f_{k}^{*} (\pi_{2})^{(n)}(U) f_{k}f_{k}^{*} (I_{n} \otimes m_{2}) (\pi_{2})^{(n)} (U^{*}) f_{k}f_{k}^{*}x^{(2)}, x^{(2)} \rangle + \epsilon \notag \\
&=&\sum\limits_{k=1}^{n} \langle f_{k}^{*} (\pi_{2})^{(n)}(U) f_{k} (I_{n} \otimes m_{2}) f_{k}^{*} (\pi_{2})^{(n)} (U^{*})f_{k}f_{k}^{*} x^{(2)}, f_{k}^{*}x^{(2)} \rangle + \epsilon. \label{46}
\end{eqnarray}
For each $k \in \{ 1, 2, \dots, n \}$, let \begin{equation} y_{k}=f_{k}^{*}x^{(2)}\quad \text {  and } \quad a_{k}=f_{k}^{*} (\pi_{2})^{(n)}(U) f_{k}\label{46.1}\end{equation}
By the choice of $f_k$, we have
$$
a_k=(e_{11}\otimes q_2 )a_k (e_{11}\otimes q_2 ) \quad \text { and } \quad
\|a_k\|=\|f_{k}^{*} (\pi_{2})^{(n)}(U) f_{k}\|\le \|f_{k}^{*} \|\| f_{k}\|= 1.
$$
 Denote $a_{k}=e_{11}\otimes a'_{k}$, where $a'_{k} \in  \pi_{2}(\mathcal{M}_{c_{1}})$ is a contraction. As  $\pi_{2}(\mathcal{M}_{c_{1}})$ is of  type II$_1$, we can find
 {\iffalse find  an element  $b_{k}$  in the unit ball of $\mathcal{M}_{c_{1}}$ such that $\pi_{2}(b_{k})=a'_{k}$. Since $\mathcal{M}_{c_{1}}$ is of type II$_1$ von Neumann algebra, $b_{k}$ is an average of two unitaries in $\mathcal{M}_{c_{1}}$. Therefore there exist\fi}
 two unitaries $u_{k}$ and $v_{k}$ in $\pi_{2}(\mathcal{M}_{c_{1}})$ such that $a'_{k}=\frac{u_{k}+v_{k}}{2}$. So
 \begin{equation}
 a_k=\frac 1 2 (e_{11}\otimes (u_k+v_k)), \qquad \forall \ k=1,\cdots, n. \label{46.5}
 \end{equation}
 By inequality (\ref{38}),
\begin{eqnarray}
(e_{11} \otimes u_{k}) (I_{n} \otimes m_{2})(e_{11} \otimes u_{k})^{*} \leq c^{2}(e_{11} \otimes m_{2}) \label{47}
\end{eqnarray}
and
\begin{eqnarray}
(e_{11} \otimes v_{k}) (I_{n} \otimes m_{2})(e_{11} \otimes v_{k})^{*} \leq c^{2} (e_{11} \otimes m_{2}) \label{48}
\end{eqnarray}

It follows from   (\ref{46}),  (\ref{46.1}) and (\ref{46.5}) that
\begin{align*}
 &\langle (\pi_{2})^{(n)}(U) (I_{n} \otimes m_{2}) (\pi_{2})^{(n)}(U^{*}) x^{(2)}, x^{(2)} \rangle \notag \\
\le&\sum\limits_{k=1}^{n} \langle f_{k}^{*} (\pi_{2})^{(n)}(U) f_{k} (I_{n} \otimes m_{2}) f_{k}^{*} (\pi_{2})^{(n)} (U^{*})f_{k}f_{k}^{*} x^{(2)}, f_{k}^{*}x^{(2)} \rangle + \epsilon\notag\\
 =& \frac{1}{4} \sum\limits_{k=1}^{n}\langle(e_{11} \otimes (u_{k}+v_{k})) (I_{n} \otimes m_{2}) (e_{11} \otimes (u_{k}+v_{k}))^{*}y_{k}, y_{k} \rangle +\epsilon \notag   \\
 \leq& \frac{1}{4} \sum\limits_{k=1}^{n}\langle(e_{11} \otimes (u_{k}+v_{k})) (I_{n} \otimes m_{2}) (e_{11} \otimes (u_{k}+v_{k}))^{*}y_{k}, y_{k} \rangle \notag \\
 &+ \frac{1}{4} \sum\limits_{k=1}^{n}\langle (e_{11} \otimes (u_{k}-v_{k})) (I_{n} \otimes m_{2}) (e_{11} \otimes (u_{k}-v_{k}))^{*}y_{k}, y_{k} \rangle + \epsilon \notag \\
 =& \frac{1}{2} \sum\limits_{k=1}^{n}\langle (e_{11} \otimes u_{k}) (I_{n} \otimes m_{2}) (e_{11} \otimes u_{k})^{*}y_{k}, y_{k} \rangle \notag \\
 &+ \frac{1}{2} \sum\limits_{k=1}^{n}\langle (e_{11} \otimes v_{k}) (I_{n} \otimes m_{2}) (e_{11} \otimes v_{k})^{*}y_{k}, y_{k} \rangle + \epsilon. \notag \\
\end{align*}

Thus by (\ref{47}) and (\ref{48}),
\begin{align}
 &\langle (\pi_{2})^{(n)}(U) (I_{n} \otimes m_{2}) (\pi_{2})^{(n)}(U^{*}) x^{(2)}, x^{(2)} \rangle \notag \\
 \leq& c^{2} \sum\limits_{k=1}^{n} \langle (e_{11}\otimes m_{2})y_{k}, y_{k} \rangle+ \epsilon \notag \\
 =&c^{2}\sum\limits_{k=1}^{n}\langle f_{k}(I_{n} \otimes m_{2} )f_{k}^{*}x^{(2)}, x^{(2)} \rangle +\epsilon. \label{49}
\end{align}

Since $f_{k}$ commutes with $I_{n} \otimes m_{2}$, we obtain from (\ref{49}) that
\begin{eqnarray}
\langle (\pi_{2})^{(n)}(U) (I_{n} \otimes m_{2}) (\pi_{2})^{(n)}(U^{*}) x^{(2)}, x^{(2)} \rangle &=& c^{2}\sum\limits_{k=1}^{n}\langle f_{k}f_{k}^{*}(I_{n} \otimes m_{2} )x^{(2)}, x^{(2)} \rangle +\epsilon \notag\\
&=&c^{2} \langle (I_{n} \otimes m_{2})x^{(2)}, x^{(2)} \rangle +\epsilon. \label{50}
\end{eqnarray}
Since $x$ and $\epsilon$ were arbitrarily chosen, (\ref{50}) implies that
\begin{eqnarray}
(\pi_{2})^{(n)}(U) (I_{n} \otimes m_{2}) (\pi_{2})^{(n)}(U^{*}) \leq c^{2} (I_{n} \otimes m_{2} ). \label{51}
\end{eqnarray}

Hence by inequalities (\ref{40}), (\ref{41}) and (\ref{51}), we obtain
\begin{eqnarray}
\pi^{(n)}(U)(I_{n} \otimes m) \pi^{(n)}(U^{*}) \leq c^{2}(I_{n} \otimes m), \qquad \text { for all unitary } \ U\in  M_n(\mathcal M). \label{52}
\end{eqnarray}

Since $m \in W$, it follows from (\ref{32}) and (\ref{35}) that $t^{2} \leq c^{2}m$ and $m \leq c^{2}t^{2}$. Therefore (\ref{52}) implies
$$\pi^{(n)}(U)(I_{n} \otimes t^{2}) \pi^{(n)}(U^{*}) \leq c^{6}(I_{n} \otimes t^{2}), \qquad \text { for all unitary } \ U\in   M_n(\mathcal M).$$
It follows that $$
\|\phi||_n\le \|\phi \|^3  \qquad \text{for all } \ n\in \Bbb N,
$$ and $$\Vert \phi \Vert_{cb} \leq \Vert \phi \Vert^{3}.$$
The proof is complete.
\end{proof}

The class of C$^*$-algebras in the following theorem was considered in \cite{HS} and it was shown in \cite{HS} such C$^*$-algebras have similarity degree no more than 11. The following theorem  gives a better estimation.
\begin{theorem}\label{5.3}
Suppose $\mathcal{A}$ is a separable unital C$^*$-algebra
satisfying
\begin{enumerate} \item[] Condition (G):
if $\pi$ is a unital $*$-representation of $\mathcal{A}$ on a Hilbert space
$  H$ such that $\pi(\mathcal{A})''$ is a type  II$_{1}$
factor, then $\pi(\mathcal{A})''$ has Property $\Gamma$, where
$\pi(\mathcal{A})''$ is the von Neumann algebra generated by
$\pi(\mathcal{A})$ in $B( H)$.\end{enumerate} Then
$d(\mathcal{A}) \leq 3.$ Moreover, if $\mathcal A$ is non-nuclear, then $d(\mathcal{A}) = 3.$
\end{theorem}
\begin{proof} Suppose $\mathcal A$ is a separable unital C$^*$-algebra
satisfying Condition (G):
{\em If $\pi$ is a unital $*$-representation of $\mathcal{A}$ on a Hilbert space
$  K_1$ such that $\pi(\mathcal{A})''$ is a type  II$_{1}$
factor, then $\pi(\mathcal{A})''$ has Property $\Gamma$, where
$\pi(\mathcal{A})''$ is the von Neumann algebra generated by
$\pi(\mathcal{A})$ in $B( K_1)$. }

Suppose $\phi$  is a unital bounded homomorphism of  $\mathcal{A}$ on  a Hilbert space $K$.

Since we just need to show that $\Vert \phi \Vert_{cb} \le \Vert \phi \Vert^{3}$, we can focus on   $\phi(\mathcal A)$ contained in the C$^*$-subalgebra of $B(K)$ generated by $\phi(\mathcal{A})$, which is a separable C$^*$-algebra. Applying the GNS construction to a faithful state on the C$^*$-subalgebra generated by $\phi(\mathcal{A})$ if necessary, we may just assume that $K$ is a separable Hilbert space.

 From Lemma \ref{2.3}, $\phi$ can be extended to a bounded unital homomorphism $\tilde{\phi}:\mathcal{A}^{\sharp \sharp} \to B(H)$ that is $\sigma(\mathcal{A}^{\sharp\sharp}, \mathcal{A}^{\sharp}) \to \sigma(B(H), B(H)_{\sharp})$ continuous. Moreover  $\Vert \tilde{\phi} \Vert = \Vert \phi \Vert$ and $\Vert {\phi} \Vert_{cb} \le \Vert \tilde{\phi} \Vert_{cb}$.

 Let $\mathcal I=ker(\tilde \phi)$. Thus $\mathcal I$ is a two-sided ideal of $\mathcal A^{\sharp\sharp}$ and is closed in the  $\sigma(\mathcal{A}^{\sharp\sharp}, \mathcal{A}^{\sharp})$ topology. From Proposition 1.10.5 in \cite{Sa}, it follows that
 there is a   projection $p$ in the center of ${\mathcal{A}^{\sharp\sharp}}$
 such that $\mathcal I=p{\mathcal{A}^{\sharp\sharp}}$. Define $\tilde\phi_p:(I-p){\mathcal{A}^{\sharp\sharp}}\rightarrow B(K)$ as follows
 $$
 \tilde\phi_p (a)=\tilde\phi(a), \qquad \forall \ a\in (I-p){\mathcal{A}^{\sharp\sharp}}.
 $$
 Then $\tilde\phi_p$ is a unital injective homomorphism, which is $\sigma(\mathcal{A}^{\sharp\sharp}, \mathcal{A}^{\sharp})$ to $\sigma(B(H), B(H)_{\sharp})$ continuous. Moreover, $\Vert \tilde{\phi} \Vert = \Vert \tilde{\phi}_p \Vert$ and $\Vert \tilde{\phi} \Vert_{cb} = \Vert \tilde{\phi}_p \Vert_{cb}$.

We claim that $(I-p){\mathcal{A}^{\sharp\sharp}}$ is countably decomposable. Assume $\{P_\lambda\}_{\lambda\in\Lambda}$ is a family of orthogonal projections in $(I-p){\mathcal{A}^{\sharp\sharp}}$ with sum $I-p$. Let $\mathcal B$ be the von Neumann subalgebra generated by $\{P_\lambda\}_{\lambda\in\Lambda}$ in $(I-p){\mathcal{A}^{\sharp\sharp}}$. Obviously, Kadision's Similarity Problem for   $\mathcal B$ has an affirmative answer. In other words, there exists a positive invertible operator $T$ in $B(K)$ such that $T \tilde{\phi}_p(\cdot)T^{-1}$ is a $*$-homomorphism from $\mathcal B$ to $B(K)$, whence $ \{ T \tilde{\phi}_p(P_\lambda)T^{-1} \}_{\lambda\in\Lambda}$ is a family of orthogonal projections in $B(K)$. From the facts that $K$ is separable and $\tilde{\phi}_p$ is injective, we conclude that $\Lambda$ is a countable set and $(I-p){\mathcal{A}^{\sharp\sharp}}$ is countably decomposable.

Since $\mathcal A$ embeds into $\mathcal A^{\sharp\sharp}$ as a  $\sigma(\mathcal{A}^{\sharp\sharp}, \mathcal{A}^{\sharp})$-dense C$^*$-subalgebra, we know that $(I-p){\mathcal{A}^{\sharp\sharp}}$ is countably generated as a von Neumann algebra. Combining  with the  result that $(I-p){\mathcal{A}^{\sharp\sharp}}$ is countably decomposable, we obtain that $(I-p){\mathcal{A}^{\sharp\sharp}}$ has a separable predual. Without loss of generality, we might assume that  $(I-p){\mathcal{A}^{\sharp\sharp}}$  acts on a separable Hilbert space $H$.

Suppose $(I-p){\mathcal{A}^{\sharp\sharp}}$ has a type decomposition
$$(I-p){\mathcal{A}^{\sharp\sharp}}= \mathcal{M}_{1} \oplus \mathcal{M}_{c_{1}} \oplus \mathcal{M}_{c_{\infty}}  \oplus \mathcal{M}_{\infty},$$
where $\mathcal{M}_{c_{1}}$ is the type  II$_{1}$ central summand. We can assume that $\mathcal{M}_{c_{1}}$  acts on a separable Hilbert space $H_1$.

Notice that if $\mathcal{M}_{c_{1}}=0$, then the proof of Theorem \ref{mainthm2} will be simplified and
$$\Vert \phi \Vert_{cb} \leq \Vert \tilde{\phi} \Vert_{cb} \leq \Vert \tilde{\phi} \Vert^{3} = \Vert \phi \Vert^{3}.$$
In the following we will assume that $\mathcal{M}_{c_{1}}$ is nonvanishing and we show that $\mathcal{M}_{c_{1}}$ is a type II$_1$ von Neumann algebra with Property $\Gamma$. Let $\mathcal{M_{c_{1}}}=\int_{X} \bigoplus M_{s} d \mu$ and $H_{1} = \int_{X} \bigoplus H_{s} d \mu$ be the direct integral decompositions of $\mathcal{M}_{c_{1}}$ and $H_{1}$ relative to the center of $\mathcal{M}_{c_{1}}$. Thus $\mathcal M_s$ is a type II$_1$ factor for almost $s\in X$. Note  that $\mathcal A$ embeds into $\mathcal A^{\sharp\sharp}$ as a  $\sigma(\mathcal{A}^{\sharp\sharp}, \mathcal{A}^{\sharp})$-dense separable C$^*$-subalgebra. Also  note $p$ is a central projection of $\mathcal A^{\sharp\sharp}$. By Theorem 14.1.13 in \cite{KR1}, there is a unital $*$-homomorphism $\psi_s$ from $(I-p)\mathcal A$ to $\mathcal M_s$ such that $\psi_s((I-p)\mathcal A)$ generates $\mathcal M_s$ as a von Neumann algebra for almost $s\in X$. Let $\hat \psi_s:\mathcal A\rightarrow \mathcal M_s$ be such that $\hat \psi_s(a)=\psi_s((I-p)a)$ for all $a\in\mathcal A$. Then $\hat\psi_s$ is a $*$-homomorphism from  $\mathcal A$ into $\mathcal M_s$ such that  $\hat \psi_s(\mathcal A)$ generates $\mathcal M_s$ as a von Neumann algebra, for $s\in X$ almost everywhere. Since $\mathcal A$ satisfies condition (G), $\mathcal M_s$ is a type II$_1$ factor with Property $\Gamma$, for $s\in X$ almost everywhere. From Proposition 3.12 in \cite{QS}, it induces that $\mathcal{M}_{c_{1}}$  is a type II$_1$ von Neumann algebra with Property $\Gamma$.

 Since $\tilde\phi_p:(I-p){\mathcal{A}^{\sharp\sharp}}\rightarrow B(K)$ is a unital   homomorphism, which is $\sigma(\mathcal{A}^{\sharp\sharp}, \mathcal{A}^{\sharp})$ to $\sigma(B(H), B(H)_{\sharp})$ continuous, by Theorem \ref{mainthm2}, we have that $\Vert \tilde{\phi}_p \Vert_{cb}\le \Vert \tilde{\phi}_p\Vert^3$, whence $$\Vert  {\phi}  \Vert_{cb}\le \Vert \tilde {\phi}  \Vert_{cb}=    \Vert \tilde {\phi}_p  \Vert_{cb} \le \Vert \tilde{\phi}_p\Vert^3 =\Vert \tilde{\phi} \Vert^3   = \Vert  {\phi} \Vert^3$$ and $d(\mathcal A)\le 3$.

It was shown in \cite{Pi3} that the similarity degree of a C$^*$-algebra is less than or equal to 2 if and only if the C$^*$-algebra  is nuclear. Notice that the similarity degree is always an integer. Since $d(\mathcal A)\le 3$, we know if $\mathcal A$   is non-nuclear, then $d(\mathcal{A}) = 3$.
 The proof is complete.
\end{proof}

 \begin{definition}
 A unital C$^*$-algebra $\mathcal A$ is said to have Property c$^*$-$\Gamma$ if it satisfies Condition (G) in Theorem \ref{5.3},  as follows: {\em If $\pi$ is a unital $*$-representation of $\mathcal{A}$ on a Hilbert space
$  H$ such that $\pi(\mathcal{A})''$ is a type  II$_{1}$
factor, then $\pi(\mathcal{A})''$ has Property $\Gamma$, where
$\pi(\mathcal{A})''$ is the von Neumann algebra generated by
$\pi(\mathcal{A})$ in $B( H)$. }
 \end{definition}

\begin{corollary}\label{5.4}
Let $\mathcal{A}$ be a separable unital C$^*$-algebra. Suppose
$\mathcal{B} $ is a nuclear separable unital C$^*$-algebra with no
finite dimensional representations. Then $\mathcal{A}
\otimes_{min} \mathcal{B}$ has Property c$^*$-$\Gamma$ and $d(\mathcal{A}
\otimes_{min} \mathcal{B}) \leq 3$, where $\mathcal{A} \otimes_{min}
\mathcal{B}$ is the minimal tensor product of $\mathcal A$ and
$\mathcal B$.
\end{corollary}
\begin{proof}
By Theorem \ref{5.3}, we only need to prove the following statement: {\em if $\pi$ is a representation of
$\mathcal{A} \otimes_{min} \mathcal{B}$ on a Hilbert space $H$ and
$\pi(\mathcal{A} \otimes_{min} \mathcal{B})''$ is a type II$_{1}$
factor, then $\pi(\mathcal{A} \otimes_{min} \mathcal{B})''$ has
Property $\Gamma$.}

Assume that $\pi$ is a representation of
$\mathcal{A} \otimes_{min} \mathcal{B}$ on a Hilbert space $H$ and
$\mathcal{M}=\left(\pi(\mathcal{A} \otimes_{min} \mathcal{B})\right)''$ is a type  II$_{1}$ factor.

  Since $I_{\mathcal A}\otimes\mathcal{B}$ is
a nuclear C$^*$-algebra with no finite dimensional representation, we get that
$\pi(I_{\mathcal A}\otimes\mathcal{B})''$ is an infinite dimensional hyperfinite von
Neumann algebra. By the fact that $\pi(I_{\mathcal A}\otimes\mathcal{B})$ commutes with
$\pi(\mathcal{A} \otimes I_{\mathcal{B}})$, it follows that  $\pi(I_{\mathcal A}\otimes\mathcal{B})''$ is a hyperfinite type II$_{1}$
factor.

Note that (i) $\mathcal{M}$ is the type  II$_{1}$ factor generated
by commuting C$^*$-subalgebras $\pi(\mathcal{A} \otimes I_{\mathcal{B}})$ and
$\pi(I_{\mathcal A}\otimes\mathcal{B})$, and (ii) $\pi(I_{\mathcal A}\otimes\mathcal{B})''$ is a hyperfinite type II$_{1}$
factor. A standard argument shows  that $\mathcal{M}$ has Property
$\Gamma$ and the proof of the   corollary is complete.
\end{proof}

Since the Jiang-Su algebra $\mathcal{Z}$ (see definition in \cite{JS1}) is nuclear, simple and infinite dimensional, we obtain the next corollary directly.

\begin{corollary}\label{5.5}
If a separable, non-nuclear, unital C$^*$-algebra $\mathcal{A}$ is $\mathcal{Z}$-stable, then $d(\mathcal{A}) = 3$.
\end{corollary}

\begin{example}\label{example}
Let $F_2$ be a non-abelian free group on two generators $a, b$ and $C^*_r(F_2)$ be the reduced C$^*$-algebra of free group factor on two standard generators $\lambda(a)$, $\lambda(b)$. Let $Aut(C^*_r(F_2))$ be the automorphism group of $C^*_r(F_2)$.
Let $\theta$ be an irrational number.
Let $\Bbb  Z$ be the  group of integers and $g $ be a generator of $\Bbb Z$.
Let $$\alpha: \ \Bbb Z\rightarrow Aut(C^*_r(F_2))$$ be a group homomorphism from $\Bbb Z$ to $Aut(C^*_r(F_2))$ defined by the following action:
$$
\alpha(g)(\lambda(a))=e^{  {2\pi i}\cdot \theta} \lambda(a)\qquad \text{and} \qquad \alpha(g)(\lambda(b))=e^{  {2\pi i}\cdot \theta}  \lambda(b) .
$$
Let $C^*_r(F_2) \rtimes_\alpha \Bbb Z$ be the reduced crossed product of $C^*_r(F_2)$ by an action $\alpha$ of   $\Bbb Z$. Then $C^*_r(F_2) \rtimes_\alpha \Bbb Z$ has Property c$^*$-$\Gamma$ (see \cite{HS}) and $d(C^*_r(F_2) \rtimes_\alpha \Bbb Z)=3.$
\end{example}

\end{document}